\documentclass{article}
\usepackage[utf8]{inputenc}
\usepackage[T1]{fontenc}
\usepackage{url}

\usepackage[pdftex,breaklinks,pdfpagemode={None},pdfstartview={FitH},
            pdfview={FitH},colorlinks,linkcolor={black},
            citecolor={black}, urlcolor={black}]{hyperref}
\usepackage{graphicx} 
\usepackage{amsthm}
\usepackage{amssymb,amsmath,xcolor,colortbl}
\usepackage{bbm}
\usepackage{fdsymbol}
\usepackage{bm}
\usepackage{enumitem}
\usepackage[sort&compress,numbers]{natbib}
\usepackage[singlelinecheck=false]{caption}
\usepackage[section]{placeins} 
\usepackage{float}

\newtheorem{theorem}{Theorem}[section]

\newtheorem{lemma}[theorem]{Lemma}
\newtheorem{conjecture}{Conjecture}
\newtheorem{defin}[theorem]{Definition}

\usepackage{tikz}
\usetikzlibrary{graphs,decorations.pathmorphing, decorations.pathreplacing, angles, quotes, shapes, arrows, topaths, calc, patterns, backgrounds, positioning, arrows, shapes.multipart}

\tikzset{stt/.style = {circle,draw = teal!100,fill = teal!100,inner sep = 1.5pt, outer sep = 2pt}}
\tikzset{stt1/.style = {circle,draw=black,thick,inner sep = 1.5pt, outer sep = 2pt}}  
\tikzset{stw/.style = {circle,thick,draw = black,inner sep = 0pt, outer sep = 2pt, minimum size = 13pt}}
\tikzset{dot/.style = {circle,fill = black,inner sep = .7pt}}
\tikzset{dot2/.style = {circle,fill = black,inner sep = .9pt,outer sep = 2pt}}
\tikzset{ellip1/.style ={draw,thick,ellipse, minimum width=2cm, minimum height=1cm,outer sep =3pt}}
\tikzset{ellip2/.style ={draw,thick,ellipse, minimum width=1.5cm, minimum height=1cm,outer sep =3pt}}
\tikzset{arr/.style = {thick,-latex}}
\tikzset{edg/.style = {thick}}

\newcommand{\email}[1]{\href{mailto:#1}{\tt#1}}

\newcommand{\ed}{\hspace{-1.4pt}\bm{-}\hspace{-1.4pt}}
\newcommand{\ra}{\hspace{-2.3pt}\rightarrow\hspace{-2.3pt}}
\newcommand{\la}{\hspace{-2.3pt}\leftarrow\hspace{-2.3pt}}
\newcommand{\edb}{\bm{-}}
\newcommand{\rab}{\rightarrow}
\newcommand{\lab}{\leftarrow}
\newcommand{\MB}{\textup{MB}}
\newcommand{\AHM}{\textup{AHM}}
\newcommand{\EB}{\textup{EB}}
\newcommand{\BCW}{\textup{BCW}}
\newcommand{\smp}{\raisebox{0.25ex}{\scalebox{0.7}{+}}}
\newcommand{\sh}{\cellcolor{gray!50}}
\newcommand{\cell}[3]{$#1_{\raisebox{0.6ex}{\scalebox{0.5}{#2}}}^{\raisebox{-0.6ex}{\scalebox{0.5}{#3}}}$}
\newcommand{\cellb}[3]{\cell{\bm{#1}}{#2}{#3}}
\newcommand{\cellbb}[3]{\cell{\bm{#1}}{#2}{{#3}}} 
\newcommand{\cellsh}[3]{\sh\cell{#1}{#2}{#3}}
\newcommand{\cellshb}[3]{\sh\cellb{#1}{#2}{#3}}
\newcommand{\cellshbb}[3]{\sh\cellbb{#1}{#2}{#3}}

\newcommand{\utA}{\begin{tikzpicture}[baseline=0ex, scale=0.5]
    \node[inner sep = 2pt] (a) at (0,0){\scriptsize{$a$}}; 
    \node[inner sep = 2pt] (w) at (1.5,0.5){$_{w}$}; 
    \node[inner sep = 3pt] (b) at (3,0){$_{b}$}; 
    \draw[-latex] (a)--(b);\draw[-latex] (a)--(w);\draw[-latex] (w)--(b);
\end{tikzpicture}}
\newcommand{\utB}{\begin{tikzpicture}[baseline=0ex, scale=0.5]
    \node[inner sep = 2pt] (a) at (0,0){\scriptsize{$a$}}; 
    \node[inner sep = 2pt] (w) at (1.5,0.5){$_{w}$}; 
    \node[inner sep = 3pt] (b) at (3,0){$_{b}$}; 
    \draw[-latex] (a)--(b);\draw (a)--(w);\draw[-latex] (w)--(b);
\end{tikzpicture}}
\newcommand{\utC}{\begin{tikzpicture}[baseline=0ex, scale=0.5]
    \node[inner sep = 2pt] (a) at (0,0){\scriptsize{$a$}}; 
    \node[inner sep = 2pt] (w) at (1.5,0.5){$_{w}$}; 
    \node[inner sep = 3pt] (b) at (3,0){$_{b}$}; 
    \draw[-latex] (a)--(b);\draw[-latex] (a)--(w);\draw[] (w)--(b);
\end{tikzpicture}}

\definecolor{azulito}{rgb}{0.9,0.9,1.0} 

\date{}
\begin{document}
\title{New mixed cages, bounds and \\computational results.}

\author{C. M. De la Cruz$^{1\textup{--}5}$ \and M. A. Pizaña$^{1\textup{--}4}$}

\maketitle

\footnotetext[1]{Universidad Autónoma Metropolitana. \href{mailto:clau_mar@ciencias.unam.mx}{\tt clau\_mar@ciencias.unam.mx}, \email{mpizana@gmail.com}}
\footnotetext[2]{Partially supported by PAPIIT-M\'exico, grant IN113324}
 \footnotetext[3]{Partially supported by SECIHTI, grant CBF 2023-2024-552.}
\footnotetext[4]{Partially supported by CONAHCYT, grant A1-S-45528.}
\footnotetext[5]{Partially supported by CONAHCYT, grant 799875.}
\renewcommand{\thefootnote}{\arabic{footnote}}

\begin{abstract}
  An $(r,z,g)$-mixed cage is a mixed graph that is $r$-regular in
  edges, $z$-regular in arcs, of girth $g$ and of minimum order
  $n=n(r,z,g)$.  We give a general lower bound for $n(r,z,g)$. In the
  case $g=3$, we determine the exact value $n(r,z,3)=r+2z+1$. For
  $g=4$ we give two improved lower bounds and we determine the exact
  value $n(r,z,4)=2r+4z$ for $r\geq 2z-2$.  We also present
  computational results, obtaining mixed cages and computational lower
  bounds by means of exhaustive backtracking searches with symmetry
  reductions. Moreover, we give an upper bound for $n(1,z,g)$ that
  matches all our computed values.
\end{abstract}

\section{Introduction}\label{sec:intro}

The cage problem has been extensively studied \cite{EJ13} since 1947
and it consists of finding an $r$-regular graph of girth $g$ with the
minimum number of vertices, which is called an $(r,g)$-cage. Cages
were introduced by Tutte in \cite{Tut47} and Sachs established their
existence for all $r\geq 2$ and $g\geq 3$ \cite{Sac63}.

The order of an $(r,g)$-cage, $n(r,g)$, is bounded below by the
\emph{Moore bound}, $\MB(r,g)$ (as attributed by Hoffman and Singleton
in \cite{HS60}, although the version for even girth seems to appear
for the first time in Erd\H{o}s and Sachs \cite{ES63}, who credit a
method of Kàrteszi \cite{Kar60}). This bound is due to the fact that every
$(r,g)$-graph must contain a \emph{Moore tree}, which is an
$r$-regular (save for leaves) tree of diameter $g-1$ and maximum
order. Numerically, the Moore bound is as follows:

$$\MB(r,g) = \left\{\begin{array}{rcll}
1+r\cdot\sum_{k=0}^{(g-3)/2}(r-1)^{k} & = &\frac{r(r-1)^{(g-1)/2}-2}{r-2} & \textup{when $g$ is odd.}\\
&&&\\
2\cdot\sum_{k=0}^{(g-2)/2}(r-1)^{k} & = &\frac{2(r-1)^{g/2}-2}{r-2} & \textup{when $g$ is even.}
\end{array}\right.
$$

\noindent (Both closed forms above assume $r>2$; for $r=2$ the
corresponding sums must be used, giving $\MB(2,g)=g$.)

There are several variants of the cage problem in the literature,
including almost regular cages~\cite{AV07}, biregular
cages~\cite{AABLL13,EJ16}, biregular bipartite
cages~\cite{AJRS22,FRJ19}, vertex-transitive cages~\cite{JS11}, Cayley
cages~\cite{EJS13}, weighted cages~\cite{ADMP26}, directed
cages~\cite{BCW70} and mixed
cages~\cite{AHM19,ADG22,Exo23,JJ25,AM26,AM26a, JJKP25}. Here, we are
interested in this last variant. We find it convenient to collect the
relevant results in Tables~\ref{tab:ref1} and~\ref{tab:ref2}, so we
can refer to individual results by the label in the Id column. For
easier comparison, several of our results in this paper appear in
Table~\ref{tab:ref2} as [{\sc dp1}]--[{\sc dp13}].  In these tables,
some of the original results were fused and a few typos corrected.

Mixed graphs admit both edges and arcs. An \emph{$(r,z,g)$-mixed
  graph} is a mixed graph that is $r$-regular in edges, $z$-regular in
arcs and has girth $g$: every vertex is incident with $r$ edges, with
$z$ outgoing arcs and with $z$ incoming arcs; the girth is the length
of the shortest cycle. An \emph{$(r,z,g)$-mixed cage} is an
$(r,z,g)$-mixed graph with minimum number of vertices, and its order
is denoted by $n(r,z,g)$. This variant was introduced by
Araujo-Pardo, Hernández-Cruz and Montellano-Ballesteros in
\cite{AHM19} where, among other results,
the authors proved the existence of mixed cages and provided a lower
bound for the order of $(r,1,g)$-mixed cages. This bound is known as
the \emph{AHM bound}, $\AHM(r,1,g)$.

In this work, we extend the AHM bound for $z\geq 1$, obtaining the
\emph{extended bound} $\EB(r,z,g)$, and hence every $(r,z,g)$-mixed
graph that attains this bound is a mixed cage. We apply this bound to
the study of mixed cages of girth $3$ and $4$: we show that
$n(r,z,3)=r+2z+1$ for all $r,z\geq 1$, and for $g=4$ we establish that
$n(r,z,4)=2r+4z$ when $r\geq 2z-2$ and that $n(r,z,4)\geq 3r+2z+2$
when $r\leq 2z-2$ and $(r,z)\neq (1,2)$.  The previous equality for
$r\geq 2z-2$ uses an upper bound by Jajcayov\'a and Jajcay in
\cite{JJ25} (see [{\sc jj1}] in Table~\ref{tab:ref2}) and generalizes
a result from \cite{AHM19} (see [{\sc ahm2}] in Table~\ref{tab:ref1})
where the equality was only claimed for $z=1$. Moreover, we also
provide an upper bound for the order of $(1,z,g)$-mixed cages ([{\sc
  dp2}], [{\sc dp3}]), which generalizes previous results by
Araujo-Pardo, Hernández-Cruz and Montellano-Ballesteros \cite{AHM19}
([{\sc ahm6}], [{\sc ahm7}]) and by Exoo \cite{Exo23} ([{\sc e1}],
[{\sc e2}]), and which improves the bounds provided by Jajcayov\'a and
Jajcay \cite{JJ25} ([{\sc jj2}], [{\sc jj3}]).  We complement the
theoretical results with computational data obtained by adapting the
code presented in \cite{DP26,CAGES10} to mixed cages. Our new code
\cite{MTCAGES11} constructs mixed cages using exhaustive searches with
symmetry reductions. This gives rise to many new mixed cages and lower
bounds presented in Tables~\ref{tab:g=3} to~\ref{tab:g=10}. These
results are too many to be included in the format used in
Tables~\ref{tab:ref1} and~\ref{tab:ref2}, but a few of them are indeed
included there for easy comparison with some previous results.

The paper is organized as follows. In Section~\ref{sec:prelim}, we
introduce the necessary terminology and background
results. Section~\ref{sec:extend} establishes the extended bound,
$\EB(r,z,g)$. Our upper bound for $(1,z,g)$-mixed cages is presented
in Section~\ref{sec:upper}. Section~\ref{sec:g3} establishes that
$n(r,z,3)=r+2z+1$. The case $g=4$ is studied in
Section~\ref{sec:g4}. Finally, Section~\ref{sec:exp} presents our
computational results.

\begin{table}[!ht]
\noindent\resizebox{\textwidth}{!}{
\begin{tabular}{|l|l|l|l|l|}
\hline
\rowcolor{azulito}Id&Result & Hypotheses & Ref. & Year\\
\hline
\hline
{\sc bcw1}& $n\leq z(g-1)+1$ & $r=0$, $z\geq 1$ & \cite{BCW70} & 1970 \\
\hline
{\sc bcw2}& $n=2z+1$ &$r=0$, $g=3$  & \cite{BCW70} & 1970\\
\hline
{\sc bcw3}& $n\geq (5z+4)/2$ &$r=0$, $z\geq 2$, $g=4$  & \cite{BCW70} & 1970\\
\hline
{\sc bcw4}& $n=3z+1$ &$r=0$, $z=1,2,3,4$, $g=4$  & \cite{BCW70} & 1970\\
\hline
{\sc bcw5}& $n=9$ &$r=0$, $z=2$, $g=5$  & \cite{BCW70} & 1970\\
\hline
{\sc ahm1}& $n\geq \AHM$ & $z=1$  & \cite{AHM19} & 2019\\
\hline
{\sc ahm2}& $n=2r+4$ & $r\geq 1$, $z=1$, $g=4$  & \cite{AHM19} & 2019\\
\hline
{\sc ahm3}& $20\leq n\leq 28$ & $r=3$, $z=1$, $g=5$  & \cite{AHM19} & 2019\\
\hline
{\sc ahm4}& $n=(g^{2}+1)/2$ & $r=2$, $z=1$, $g=2k+1$  & \cite{AHM19} & 2019\\
\hline
{\sc ahm5}& $n = g^{2}/2$ & $r=2$, $z=1$, $g=2k$  & \cite{AHM19} & 2019\\
\hline
{\sc ahm6}& \begin{tabular}{@{}l@{}l@{}}
$n$ &$\leq 3(z+1)$ \\
&$= (z+1)(g-1)$ 
\end{tabular} & $r=1$, $z=2k+1$, $g=4$ & \cite{AHM19} & 2019\\
\hline
{\sc ahm7}& \begin{tabular}{@{}l@{}l@{}} 
$n$ &$\leq 3z+2$ \\
&$= z(g-1)+2$
\end{tabular}& $r=1$, $z=2k \geq 2$, $g=4$ & \cite{AHM19} & 2019\\
\hline
{\sc adg1}& $n\geq \MB+2z$ & None  & \cite{ADG22} & 2022\\
\hline
{\sc adg2}& $n\leq 2pr$ & \begin{tabular}{@{}l@{}}
$1\leq r\leq p$, $z\neq 2$, $g=5$, \\
$4z+1\leq p \leq 5z$ 
\end{tabular}  & \cite{ADG22} & 2022\\
\hline
{\sc adg3}& $n\leq g\cdot n(r,0,g)$ & $r\geq 2$, $1\leq z \leq n(r,0,g)$ & \cite{ADG22} & 2022\\
\hline
{\sc adg4}& $n=50$ & $r=3$, $z=10$, $g=5$  & \cite{ADG22} & 2022\\
\hline
{\sc e1}& $n=2g-2$ & $r=1$, $z=1$ & \cite{Exo23} & 2023\\
\hline
{\sc e2}& $n=2g$ & $r=1$, $z=2$ & \cite{Exo23} & 2023\\
\hline
{\sc e3}& $n\leq g^{2}$ & $r=2$, $z\leq g$ & \cite{Exo23} & 2023\\
\hline
{\sc e4}& $n\leq g^{2}-g$ & $r=2$, $z\leq g/2$, $g=2k$& \cite{Exo23} & 2023\\
\hline
{\sc e5}& $n\leq \left\lceil\frac{g}{2}\right\rceil\left\lfloor\frac{3g}{2}\right\rfloor
\sim \frac{3g^{2}}{4}$ & $r=2$, $z=2$ & \cite{Exo23} & 2023\\
\hline
{\sc e6}& $n=19$ & $r=2$, $z=2$, $g=5$  & \cite{Exo23} & 2023\\
\hline
{\sc e7}& $n=24$ & $r=3$, $z=1$, $g=5$  & \cite{Exo23} & 2023\\
\hline
{\sc e8}& $n=27$ & $r=2$, $z=2$, $g=6$  & \cite{Exo23} & 2023\\
\hline
{\sc e9}& $n=30$ & $r=3$, $z=1$, $g=6$  & \cite{Exo23} & 2023\\
\hline
{\sc e10}& $29\leq n\leq 34$ & $r=4$, $z=1$, $g=5$  & \cite{Exo23} & 2023\\
\hline
{\sc e11}& $40\leq n\leq 50$ & $r=5$, $z=1$, $g=5$  & \cite{Exo23} & 2023\\
\hline
{\sc e12}& $46\leq n\leq 48$ & $r=4$, $z=1$, $g=6$  & \cite{Exo23} & 2023\\
\hline
{\sc e13}& $66\leq n\leq 72$ & $r=5$, $z=1$, $g=6$  & \cite{Exo23} & 2023\\
\hline
{\sc e14}& $52\leq n\leq 60$ & $r=3$, $z=1$, $g=7$  & \cite{Exo23} & 2023\\
\hline
{\sc e15}& $74\leq n\leq 76$ & $r=3$, $z=1$, $g=8$  & \cite{Exo23} & 2023\\
\hline
\end{tabular}}
\caption{Results in the literature (first part). Here $r\geq 0$,
  $z\geq 1$, $g\geq 3$, $n=n(r,z,g)$, $k\geq 0$, $p$ is a prime
  number, $\MB=\MB(r,g)$ is the Moore bound and $\AHM=\AHM(r,1,g)$ is
  the AHM bound. Column ``Id'' is used for referencing the
  result.}\label{tab:ref1}
\end{table}

\begin{table}[!ht]
\noindent\resizebox{\textwidth}{!}{
\begin{tabular}{|l|l|l|l|l|}
\hline
\rowcolor{azulito}Id&Result & Hypotheses & Ref. & Year\\
\hline
\hline
{\sc jj1}& $n\leq 2r+4z$ & $r\geq 2$, $z\geq 1$, $g=4$ & \cite{JJ25} & 2025\\
\hline
{\sc jj2}& $n\leq zg$ &$r=1$, $z=2k \geq 2$ & \cite{JJ25} & 2025\\
\hline
{\sc jj3}& $n\leq (z+1)g$ &$r=1$, $z=2k+1 \geq 3$ & \cite{JJ25} & 2025\\
\hline
{\sc jj4}& $n \leq 2(q^{2}+q+1)$ &
\begin{tabular}{@{}l@{}}
$r=q+1-2z$, $z\geq 1$, $g=6$, \\$q\geq 2z+4$
\end{tabular} & \cite{JJ25} & 2025\\
\hline
{\sc jj5}& $n \leq 2(q^{3}+q^{2}+q+1)$ &
\begin{tabular}{@{}l@{}}
$r=q+1-2z$, $z\geq 1$, $g=8$, \\$q\geq 2z+6$
\end{tabular} & \cite{JJ25} & 2025\\
\hline
{\sc jj6}& $n \leq 2(q^{3}+1)(q^{2}+q+1)$ &
\begin{tabular}{@{}l@{}}
$r=q+1-2z$, $z\geq 2$, $g=12$, \\
$q\geq 2z+10$
\end{tabular} & \cite{JJ25} & 2025\\
\hline
{\sc jjkp1}& $n \leq 2q^{2}-2$ &
$r=q>7$, $z\leq \left\lfloor \frac{q-2}{4} \right\rfloor$, $g=5$ & \cite{JJKP25} & 2025\\
\hline
{\sc jjkp2}& $n \leq 2q^{2}$ &
\begin{tabular}{@{}l@{}}
$r=q=p^{\alpha}$, $z=kp^{\alpha-1}$, $g=5$, \\
 $p\geq 5$, $1\leq k \leq \left\lfloor \frac{p}{4} \right\rfloor$
\end{tabular} & \cite{JJKP25} & 2025\\
\hline
{\sc am1}& $n \leq 2q^{2}-2$ &
\begin{tabular}{@{}l@{}}
$r=q=4z+S$, $z\geq 1$, $g=5$,\\
$S\in\{2,3,\ldots, 6\}$
\end{tabular} & \cite{AM26} & 2026\\
\hline
{\sc am2}& $n \leq 4q^{2}-4$ &
$r=q$, $z=\left\lfloor \frac{q}{4} \right\rfloor$, $g=6$ & \cite{AM26a} & 2026\\
\hline
{\sc dp1}& $n\geq \EB = \AHM+2z-2$ & None & T\ref{thm:EB} & 2026\\
\hline
{\sc dp2}& $n\leq z(g-1)+2$ & $r=1$, $z=2k\geq 2$ & T\ref{thm:MCirc} & 2026\\
\hline
{\sc dp3}& $n\leq (z+1)(g-1)$ & $r=1$, $z=2k+1$ & T\ref{thm:MCirc} & 2026\\
\hline
{\sc dp4}& $n=r+2z+1$ & $g=3$ & T\ref{thm:g3} & 2026\\
\hline
{\sc dp5}& $n\geq 2r+2z+2$ & $g=4$ & T\ref{thm:EB} & 2026\\
\hline
{\sc dp6}& $n\geq 3r+2z+2$ & 
\begin{tabular}{@{}l@{}}
$r\leq 2z-2$, $g=4$, \\
$(r,z)\neq (1,2)$
\end{tabular} & T\ref{thm:2z-2} & 2026\\
\hline
{\sc dp7}& $n= 2r+4z$ & $r\geq 2z-2$, $g=4$ & T\ref{thm:2z-2} & 2026\\
\hline
{\sc dp8}& $n\geq 31$ & $r=4$, $z=1$, $g=5$  & Tb\ref{tab:g=5} & 2026\\
\hline
{\sc dp9}& $n\geq 42$ & $r=5$, $z=1$, $g=5$  & Tb\ref{tab:g=5} & 2026\\
\hline
{\sc dp10}& $n= 48$ & $r=4$, $z=1$, $g=6$  & Tb\ref{tab:g=6} & 2026\\
\hline
{\sc dp11}& $n\geq 68$ & $r=5$, $z=1$, $g=6$  & Tb\ref{tab:g=6} & 2026\\
\hline
{\sc dp12}& $n\geq 56$ & $r=3$, $z=1$, $g=7$  & Tb\ref{tab:g=7} & 2026\\
\hline
{\sc dp13}& $n= 76$ & $r=3$, $z=1$, $g=8$  & Tb\ref{tab:g=8} & 2026\\
\hline
\end{tabular}}
\caption{Results in the literature (continuation). Same conventions of
  the previous table and, besides, $\alpha\geq 1$, $q$ is a prime
  power number and $\EB=\EB(r,z,g)$ is the extended bound. In the
  column ``Ref.'', T$x$ and Tb$x$ refer, respectively, to Theorem $x$
  and Table $x$ of this paper.}\label{tab:ref2}
\end{table}

\section{Terminology and preliminaries}\label{sec:prelim}

We consider mixed graphs, which may contain edges and arcs but no
loops, and given two different vertices $x,y$ at most one edge or arc
connects them (so, no parallel edges, etc.). If $X$ and $Y$ are sets
of vertices of a mixed graph, we write $X\ed Y$ ($X\ra Y$) to indicate
that for every vertex in $X$ and every vertex in $Y$ there is an edge
(arc) connecting them. We may also write $Y\la X$ instead of $X\ra
Y$. When using this notation, we write simply $x$ instead of the
singleton $\{x\}$, hence, we may write $x\ed y$, $y\la x$, $x \ed Y$,
$X \ra y$, and so on. Note that ``$x\ed y$'' is different from the
numerical difference ``$x-y$''.
  
Given a vertex $x$, we define its \emph{neighborhood} as
$N(x) = \{y : x\ed y\}$, its \emph{in-neighborhood} as
$N^{-}(x) = \{y : x\la y\}$, and its \mbox{\emph{out-neighborhood}} as
$N^{+}(x) = \{y : x\ra y\}$. Note that, for our mixed graphs, the sets
$N(x)$, $N^{-}(x)$ and $N^{+}(x)$ are pairwise disjoint, we call this
property the \emph{disjoint condition}. The cardinalities of the sets
$N(x)$, $N^{-}(x)$ and $N^{+}(x)$ are the \emph{degree},
\emph{in-degree} and \emph{out-degree} of the vertex $x$,
respectively. A mixed graph is \emph{$(r,z)$-regular} if every vertex
has degree $r$, out-degree $z$ and in-degree $z$; if we drop this last
in-degree condition, the mixed graph is
\emph{$(r,z)$-semiregular}. Some authors prefer the terms
\emph{totally regular} and \emph{regular}, respectively. We may
occasionally use the term ``totally regular'' for emphasis and to
avoid ambiguities.

A \emph{cycle} of length $t$ in a mixed graph is a sequence of
vertices $(x_{0}, x_{1}, \ldots, x_{t})$ where $x_{0} = x_{t}$, the
vertices $x_{0}, x_{1}, \ldots, x_{t-1}$ are pairwise different, and
such that for all $i\in \{0,1,\ldots, t-1\}$ we have either
$x_{i}\ed x_{i+1}$ or $x_{i}\ra x_{i+1}$.  Such a cycle is said to be
a \emph{$t$-cycle} (necessarily $t\geq 3$). The \emph{girth} of a mixed
graph $G$, denoted $g(G)$, is the length of the shortest cycle of
$G$. We shall need to
distinguish $3$-cycles as in Figure~\ref{fig:3cycles} from
\emph{non-cyclic triangles} as in Figure~\ref{fig:untriangles}: the
former have girth 3 and the latter have infinite girth.

\begin{figure}[!ht]
  \begin{tikzpicture}[scale=.75]
    \node[stt1] (1) at (1,1) {};
    \node[stt1] (2) at (3,1) {};
    \node[stt1] (3) at (2,2) {};
    \node[stt1] (4) at (5,1) {};
    \node[stt1] (5) at (7,1) {};
    \node[stt1] (6) at (6,2) {};
    \node[stt1] (7) at (9,1) {};
    \node[stt1] (8) at (11,1) {};
    \node[stt1] (9) at (10,2) {};
    \node[stt1] (10) at (13,1) {};
    \node[stt1] (11) at (15,1) {};
    \node[stt1] (12) at (14,2) {};
    \draw [edg] (1)--(2); \draw [edg] (1)--(3);\draw [edg] (2)--(3);
    \draw [edg] (4)--(5); \draw [arr] (4)--(6);\draw [edg] (6)--(5);
    \draw [edg] (7)--(8); \draw [arr] (7)--(9);\draw [arr] (9)--(8);    
    \draw [arr] (11)--(10); \draw [arr] (10)--(12);\draw [arr] (12)--(11);    
  \end{tikzpicture}
  \caption{All 3-cycles.}
  \label{fig:3cycles}
\end{figure}
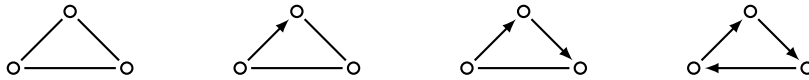

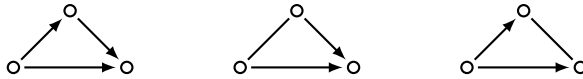
\begin{figure}[!ht]
  \begin{tikzpicture}[scale=.75]
    \node[stt1] (1) at (1,1) {};
    \node[stt1] (2) at (3,1) {};
    \node[stt1] (3) at (2,2) {};
    \node[stt1] (4) at (5,1) {};
    \node[stt1] (5) at (7,1) {};
    \node[stt1] (6) at (6,2) {};
    \node[stt1] (7) at (9,1) {};
    \node[stt1] (8) at (11,1) {};
    \node[stt1] (9) at (10,2) {};
    \draw [arr] (1)--(2); \draw [arr] (1)--(3);\draw [arr] (3)--(2);
    \draw [arr] (4)--(5); \draw [edg] (4)--(6);\draw [arr] (6)--(5);
    \draw [arr] (7)--(8); \draw [arr] (7)--(9);\draw [edg] (9)--(8);    
  \end{tikzpicture}
  \caption{All non-cyclic triangles.}
  \label{fig:untriangles}
\end{figure}

An \emph{$(r,z,g)$-mixed graph} is an $(r,z)$-regular mixed graph of
girth $g$. An \emph{$(r,z,g)$-mixed cage} is an $(r,z,g)$-mixed graph
of minimum order. The order of an $(r,z,g)$-mixed cage will be denoted
by $n(r,z,g)$. When $z=0$, mixed cages are simply cages, and when
$r=0$, they are directed cages. Therefore, in this paper we assume
$z\geq 1$ and $r\geq 1$ unless otherwise stated. 

In a similar way, semiregular mixed cages may be defined: an
\emph{$(r,z,g)$-semiregular mixed graph} is an $(r,z)$-semiregular
mixed graph of girth $g$ and an \emph{$(r,z,g)$-semiregular mixed
  cage} is an $(r,z,g)$-semiregular mixed graph of minimum order. The
order of an $(r,z,g)$-semiregular mixed cage will be denoted by
$n_{s}(r,z,g)$. Clearly $n_{s}(r,z,g) \leq n(r,z,g)$.  It is worth
mentioning that Exoo studied semiregular mixed cages in \cite{Exo23}
and hence all his lower bounds are also lower bounds for mixed
cages. On the other hand, Exoo's upper bounds, in principle, are not
necessarily upper bounds for mixed cages, however, all his
constructions are indeed (totally) regular and hence all kinds of
bounds from \cite{Exo23} are applicable to mixed cages. There are no
known examples where the strict inequality $n_{s}(r,z,g) < n(r,z,g)$
holds.

In \cite{AHM19}, Araujo-Pardo, Hernández-Cruz and
Montellano-Ballesteros gave a lower bound, \emph{$\AHM(r,1,g)$}, for
the order of the $(r,1,g)$-mixed cages. It is obtained by counting the
vertices of a mixed tree that must be contained as a subgraph in any
$(r,1,g)$-mixed graph. This mixed tree is constructed by rooting
copies of Moore trees at each internal vertex of a directed path of
length $g-1$ as in Figure~\ref{fig:AHM}: the Moore trees increase in
depth from the extremes to the center. It is easy to see that the
girth associated to these Moore trees (as in the Moore bound
$\MB(r,g)$) is given by $\gamma(i) =\min(2i+1,2g-2i-1)$, where $i$
runs from 1 to $g-2$. Hence the AHM bound may be expressed as in the
following theorem.

\begin{theorem}\textup{\cite[Thm.~2]{AHM19}}\label{thm:AHM1}
$n(r,1,g)\geq \AHM(r,1,g) = 2 + \sum_{i=1}^{g-2} \MB(r,\gamma(i))$
\end{theorem}

\begin{figure}[!ht]
\centering
\begin{tikzpicture}[scale = 0.7]

\begin{scope}[xshift=-0.4cm, yshift=3cm, scale=1.41]

\node at (3,2.5) {$a$};
\node at (9,2.5) {$b$};

\node[stt] (ini) at (1,2) {};
\node[stt] (fin) at (11,2) {};

\node[stt] (3) at (3,2) {};
\node[stt] (4) at (5,2) {};
\node[stt] (5) at (7,2) {};
\node[stt] (6) at (9,2) {};
\node[stt] (9) at (2.5,1.5) {};
\node[stt] (10) at (3,1.5) {};
\node[stt] (11) at (3.5,1.5) {};
\node[stt] (12) at (4.5,1.5) {};
\node[stt] (13) at (5,1.5) {};
\node[stt] (14) at (5.5,1.5) {};
\node[stt] (15) at (4.25,1) {};
\node[stt] (16) at (4.55,1) {};
\node[stt] (17) at (4.85,1) {};
\node[stt] (18) at (5.15,1) {};
\node[stt] (19) at (5.45,1) {};
\node[stt] (20) at (5.75,1) {};
\node[stt] (21) at (6.5,1.5) {};
\node[stt] (22) at (7,1.5) {};
\node[stt] (23) at (7.5,1.5) {}; 
\node[stt] (24) at (6.35,1) {};
\node[stt] (25) at (6.65,1) {};
\node[stt] (26) at (6.85,1) {};
\node[stt] (27) at (7.15,1) {};
\node[stt] (28) at (7.35,1) {};
\node[stt] (29) at (7.65,1) {};
\node[stt] (30) at (8.5,1.5) {};
\node[stt] (31) at (9,1.5) {};
\node[stt] (32) at (9.5,1.5) {};

\draw [arr] (ini) -- (3);
\draw [arr] (3) -- (4);
\draw [arr] (4) -- (5);
\draw [arr] (5) -- (6);
\draw [arr] (6) -- (fin);

\draw [edg] (9)--(3)--(10);
\draw [edg] (3)-- (11);

\draw [edg] (15)--(12)--(4)--(13)--(17);
\draw [edg] (12)-- (16);
\draw [edg] (13)-- (18);
\draw [edg] (4)--(14)--(19);
\draw [edg] (14)-- (20);

\draw [edg] (24)--(21)--(5)--(22)--(26);
\draw [edg] (21)-- (25);
\draw [edg] (22)-- (27);
\draw [edg] (5)--(23)--(29);
\draw [edg] (23)-- (28);

\draw [edg] (30)--(6)--(31);
\draw [edg] (6)-- (32);

\end{scope}

\begin{scope}[xshift=0cm, yshift=0cm, scale=1]

\node at (3,2.5) {$a$};
\node at (13,2.5) {$b$};

\node[stt] (ini) at (1,2) {};
\node[stt] (fin) at (15,2) {};


\node[stt] (3) at (3,2) {};
\node[stt] (4) at (5,2) {};
\node[stt] (5) at (8,2) {};
\node[stt] (6) at (11,2) {};
\node[stt] (7) at (13,2) {};


\node[stt] (10) at (2.5,1.2) {};
\node[stt] (11) at (3,1.2) {};
\node[stt] (12) at (3.5,1.2) {};

\node[stt] (13) at (4.4,1.2) {};
\node[stt] (14) at (5,1.2) {};
\node[stt] (15) at (5.6,1.2) {};
\node[stt] (16) at (4.25,0.4) {};
\node[stt] (17) at (4.55,0.4) {};
\node[stt] (18) at (4.85,0.4) {};
\node[stt] (19) at (5.15,0.4) {};
\node[stt] (20) at (5.45,0.4) {};
\node[stt] (21) at (5.75,0.4) {};

\node[stt] (22) at (6.7,1.2) {};
\node[stt] (23) at (8,1.2) {};
\node[stt] (24) at (9.3,1.2) {};
\node[stt] (25) at (6.40,0.4) {};
\node[stt] (26) at (7,0.4) {};
\node[stt] (27) at (7.70,0.4) {};
\node[stt] (28) at (8.30,0.4) {};
\node[stt] (29) at (9,0.4) {};
\node[stt] (30) at (9.6,0.4) {};
\node[stt] (31) at (6.25,-0.4) {};
\node[stt] (32) at (6.55,-0.4) {};
\node[stt] (33) at (6.85,-0.4) {};
\node[stt] (34) at (7.15,-0.4) {};
\node[stt] (35) at (7.55,-0.4) {};
\node[stt] (36) at (7.85,-0.4) {};
\node[stt] (37) at (8.15,-0.4) {};
\node[stt] (38) at (8.45,-0.4) {};
\node[stt] (39) at (8.85,-0.4) {};
\node[stt] (40) at (9.15,-0.4) {};
\node[stt] (41) at (9.45,-0.4) {};
\node[stt] (42) at (9.75,-0.4) {};

\node[stt] (43) at (10.4,1.2) {};
\node[stt] (44) at (11,1.2) {};
\node[stt] (45) at (11.6,1.2) {};
\node[stt] (46) at (10.25,0.4) {};
\node[stt] (47) at (10.55,0.4) {};
\node[stt] (48) at (10.85,0.4) {};
\node[stt] (49) at (11.15,0.4) {};
\node[stt] (50) at (11.45,0.4) {};
\node[stt] (51) at (11.75,0.4) {};

\node[stt] (52) at (12.5,1.2) {};
\node[stt] (53) at (13,1.2) {};
\node[stt] (54) at (13.5,1.2) {};

\draw [arr] (ini) -- (3);
\draw [arr] (3) -- (4);
\draw [arr] (4) -- (5);
\draw [arr] (5) -- (6);
\draw [arr] (6) -- (7);
\draw [arr] (7) -- (fin);

\draw [edg] (10)--(3)--(11);
\draw [edg] (3)-- (12);

\draw [edg] (16)--(13)--(4)--(14)--(18);
\draw [edg] (13)-- (17);
\draw [edg] (14)-- (19);
\draw [edg] (4)--(15)--(20);
\draw [edg] (15)-- (21);

\draw [edg] (31)--(25)--(22)--(5)--(23)--(28)--(38);
\draw [edg] (22)--(26)--(34);
\draw [edg] (25)-- (32);
\draw [edg] (26)-- (33);
\draw [edg] (23)--(27)--(35);
\draw [edg] (27)-- (36);
\draw [edg] (28)-- (37);
\draw [edg] (5)--(24)--(30)--(42);
\draw [edg] (24)--(29)--(39);
\draw [edg] (29)-- (40);
\draw [edg] (30)-- (41);

\draw [edg] (46)--(43)--(6)--(44)--(49);
\draw [edg] (43)-- (47);
\draw [edg] (44)-- (48);
\draw [edg] (6)--(45)--(51);
\draw [edg] (45)-- (50);

\draw [edg] (52)--(7)--(53);
\draw [edg] (7)-- (54);

\end{scope}

\end{tikzpicture}

\caption{AHM trees for $(r,1,g)=(3,1,6)$ and $(r,1,g)=(3,1,7)$.}\label{fig:AHM}
\end{figure}
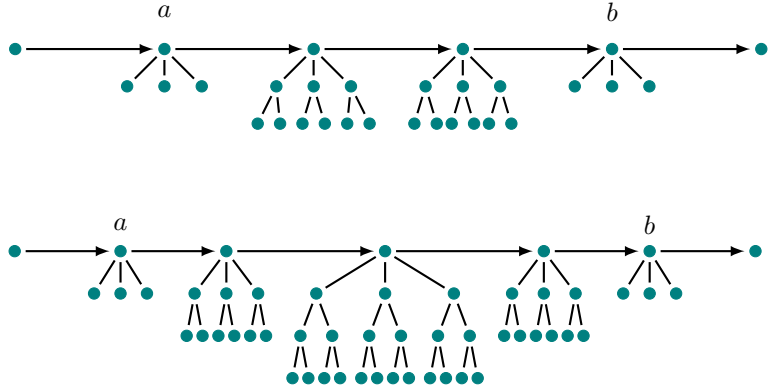 

Jajcayov\'a and Jajcay \cite{JJ25} gave a construction of an
$(r,z,4)$-mixed graph for $r\geq 2 $ and $z\geq 1$, which provides the
upper bound in the following theorem. Here we provide a slightly more
explicit construction that also works for $r=1$.

\begin{theorem}\textup{\cite{JJ25}}\label{JJ_bound}
For $r,z\geq 1$, we have $n(r,z,4) \leq 2r+4z.$
\end{theorem}
\begin{proof}
  Take $m=r+2z$ and note that $2r+4z = 2(r+2z) = 2m$. Let us define
  the following bipartite mixed graph $G$ on $2m$ vertices: take
  $V=\{x_{i}, y_{i} : i\in \mathbb{Z}_{m}\}$ and define the adjacencies
  as follows:

$$ \begin{array}{lll}
     N(x_{i}) &=& \{y_{i+\alpha}: 0\leq \alpha < r\},\\
     N^{+}(x_{i}) &=& \{y_{i+\alpha}: r\leq \alpha < r+z\},\\ 
     N^{-}(x_{i}) &=& \{y_{i+\alpha}: r+z\leq \alpha < r+2z\}.
\end{array}$$

It should be clear that
$(x_{0}, y_{r+z-1}, x_{r+z-1}, y_{r+z}, x_{0})$ is always a 4-cycle
(either $r=1$ and
$(x_{0}\ra y_{r+z-1}\ed x_{r+z-1}\ra y_{r+z}\ra x_{0})$, or $r\geq 2$
and $(x_{0}\ra y_{r+z-1}\ed x_{r+z-1}\ed y_{r+z}\ra x_{0})$). Hence
$G$ is an $(r,z,4)$-mixed graph.
\end{proof}

As usual, $\overrightarrow{C}_{n}(j_{1},j_{2},\ldots, j_{t})$ denotes
the circulant with vertices in $\mathbb{Z}_{n}$ where there is an arc
from vertex $x$ to vertex $y$ whenever
$y-x \in \{j_{1},j_{2},\ldots, j_{t}\}$. Since antiparallel arcs are
forbidden in our mixed graphs, we assume the convention that such
antiparallel arcs resulting from the previous definition are replaced
by edges. Hence, for instance, $\overrightarrow{C}_{4}(1,2)$ is a
directed 4-cycle with two edges as diagonals and
$\overrightarrow{C}_{6}(\pm 1,2)$ is a 6-cycle made of edges, with two
directed 3-cycles inside. Behzad, Chartrand and Wall in \cite{BCW70}
showed that the circulant
$\BCW(z,g) = \overrightarrow{C}_{n}(1,2,\ldots,z)$ with $n=z(g-1)+1$
is a $(0,z,g)$-mixed graph (digraph actually) and hence that
$n(0,z,g)\leq z(g-1)+1$ (see also [{\sc bcw1}]); they also conjectured
that this bound is sharp. For future reference, we transcribe the
conjecture here:

\begin{conjecture}\textup{\cite{BCW70}}\label{conj:BCW}
For $z\geq 1$ and $g\geq 2$, $n(0,z,g)= z(g-1)+1$.
\end{conjecture}

Given two mixed graphs $G$ and $H$, the \emph{composition}, $G[H]$ is
the mixed graph with vertex set $V(G)\times V(H)$ and with the
following adjacency rules:\bigskip

\begin{tabular}{lll}
$(g_{1},h_{1})\ra (g_{2},h_{2})$ & $\iff$ &
$g_{1}\ra g_{2}$ or ($g_{1}=g_{2}$ and $h_{1}\ra h_{2}$).\\
$(g_{1},h_{1})\ed (g_{2},h_{2})$ & $\iff$ & 
$g_{1}\ed g_{2}$ or ($g_{1}=g_{2}$ and $h_{1}\ed h_{2}$).
\end{tabular}

\section{Extended bound}\label{sec:extend}

A small variation of the AHM trees in Figure~\ref{fig:AHM} gives us a
new bound, valid for all $r,z,g$. Take
$\gamma(i) =\min(2i+1,2g-2i-1)$ as in the definition of the AHM bound
and define the \emph{extended bound} as
$\EB(r,z,g) = 2z + \sum_{i=1}^{g-2} \MB(r,\gamma(i)) =
\AHM(r,1,g)+2z-2$.
 
\begin{theorem}\label{thm:EB} 
  Let $z\geq 1$. An $(r,z,g)$-mixed graph must contain a subgraph as
  the one depicted in Figure~\ref{fig:EB} and hence
  $n(r,z,g)\geq \EB(r,z,g)$. In particular,
  $n(r,z,3)\geq \EB(r,z,3) = r+2z+1$ and
  $n(r,z,4)\geq \EB(r,z,4) = 2r+2z+2$. All these bounds are also valid
  for $(r,z,g)$-semiregular mixed graphs.
\end{theorem}
\begin{proof} Since we are considering $(r,z,g)$-mixed graphs, it is clear
  that all the vertices depicted in Figure~\ref{fig:EB} must have at
  least as many neighbors as depicted. All of these vertices must be
  distinct, since otherwise a cycle shorter than $g$ would be formed. 

  In the semiregular case, only the in-neighborhood of vertex $a$ in
  Figure~\ref{fig:EB} is not guaranteed. However, it should be clear
  that the average in-degree in an $(r,z,g)$-semiregular mixed graph
  is $z$, and hence there is a vertex with at least $z$
  in-neighbors. Therefore, one can construct the required subgraph
  starting with such a vertex~$a$.
\end{proof}

\begin{theorem}\label{thm:EBreg}
  Every $(r,z,g)$-semiregular mixed graph of order $\EB(r,z,g)$ is
  totally regular, and hence it is an $(r,z,g)$-mixed cage.
\end{theorem}
\begin{proof}
  If an $(r,z,g)$-semiregular mixed graph, $G$, is not totally regular, since it has
  average in-degree $z$, there must be a vertex of in-degree greater
  than $z$. Now one can construct the graph in Figure~\ref{fig:EB}
  starting with that vertex as vertex $a$, granting at least one
  additional in-neighbor to $a$. Therefore, $|G|\geq \EB(r,z,g)+1$,
  which is a contradiction.
\end{proof}

\begin{figure}[!ht]
\centering
\begin{tikzpicture}[scale = 0.7]

\begin{scope}[xshift=-0.5cm, yshift=3cm, scale=1.37]

\node at (3,2.5) {$a$};
\node at (9,2.5) {$b$};

\node at (0.5,2.02) {$z$ $\left\{\vphantom{\rotatebox{90}{mmm}}\right.$};
\node at (11.7,2.02) {$\left.\vphantom{\rotatebox{90}{mmm}}\right\}$ $z$ };

\node[dot] at (1,1.8) {};
\node[dot] at (1,2) {};
\node[dot] at (1,2.2) {};

\node[dot] at (11,1.8) {};
\node[dot] at (11,2) {};
\node[dot] at (11,2.2) {};

\node[stt] (1) at (1,2.6) {};
\node[stt] (2) at (1,1.4) {};
\node[stt] (3) at (3,2) {};
\node[stt] (4) at (5,2) {};
\node[stt] (5) at (7,2) {};
\node[stt] (6) at (9,2) {};
\node[stt] (7) at (11,2.6) {};
\node[stt] (8) at (11,1.4) {};
\node[stt] (9) at (2.5,1.5) {};
\node[stt] (10) at (3,1.5) {};
\node[stt] (11) at (3.5,1.5) {};
\node[stt] (12) at (4.5,1.5) {};
\node[stt] (13) at (5,1.5) {};
\node[stt] (14) at (5.5,1.5) {};
\node[stt] (15) at (4.25,1) {};
\node[stt] (16) at (4.55,1) {};
\node[stt] (17) at (4.85,1) {};
\node[stt] (18) at (5.15,1) {};
\node[stt] (19) at (5.45,1) {};
\node[stt] (20) at (5.75,1) {};
\node[stt] (21) at (6.5,1.5) {};
\node[stt] (22) at (7,1.5) {};
\node[stt] (23) at (7.5,1.5) {}; 
\node[stt] (24) at (6.35,1) {};
\node[stt] (25) at (6.65,1) {};
\node[stt] (26) at (6.85,1) {};
\node[stt] (27) at (7.15,1) {};
\node[stt] (28) at (7.35,1) {};
\node[stt] (29) at (7.65,1) {};
\node[stt] (30) at (8.5,1.5) {};
\node[stt] (31) at (9,1.5) {};
\node[stt] (32) at (9.5,1.5) {};

\draw [arr] (1) -- (3);
\draw [arr] (2) -- (3);
\draw [arr] (3) -- (4);
\draw [arr] (4) -- (5);
\draw [arr] (5) -- (6);
\draw [arr] (6) -- (7);
\draw [arr] (6) -- (8);

\draw [edg] (9)--(3)--(10);
\draw [edg] (3)-- (11);

\draw [edg] (15)--(12)--(4)--(13)--(17);
\draw [edg] (12)-- (16);
\draw [edg] (13)-- (18);
\draw [edg] (4)--(14)--(19);
\draw [edg] (14)-- (20);

\draw [edg] (24)--(21)--(5)--(22)--(26);
\draw [edg] (21)-- (25);
\draw [edg] (22)-- (27);
\draw [edg] (5)--(23)--(29);
\draw [edg] (23)-- (28);

\draw [edg] (30)--(6)--(31);
\draw [edg] (6)-- (32);

\end{scope}

\begin{scope}[xshift=0cm, yshift=0cm, scale=1]

\node at (3,2.5) {$a$};
\node at (13,2.5) {$b$};


\node at (0.2,2) {$z$ $\Biggl\{ $};
\node at (15.7,2) {$\Biggr\}$ $z$};

\node[dot] at (1,1.8) {};
\node[dot] at (1,2) {};
\node[dot] at (1,2.2) {};

\node[dot] at (15,1.8) {};
\node[dot] at (15,2) {};
\node[dot] at (15,2.2) {};

\node[stt] (1) at (1,2.6) {};
\node[stt] (2) at (1,1.4) {};

\node[stt] (3) at (3,2) {};
\node[stt] (4) at (5,2) {};
\node[stt] (5) at (8,2) {};
\node[stt] (6) at (11,2) {};
\node[stt] (7) at (13,2) {};

\node[stt] (8) at (15,2.6) {};
\node[stt] (9) at (15,1.4) {};

\node[stt] (10) at (2.5,1.2) {};
\node[stt] (11) at (3,1.2) {};
\node[stt] (12) at (3.5,1.2) {};

\node[stt] (13) at (4.4,1.2) {};
\node[stt] (14) at (5,1.2) {};
\node[stt] (15) at (5.6,1.2) {};
\node[stt] (16) at (4.25,0.4) {};
\node[stt] (17) at (4.55,0.4) {};
\node[stt] (18) at (4.85,0.4) {};
\node[stt] (19) at (5.15,0.4) {};
\node[stt] (20) at (5.45,0.4) {};
\node[stt] (21) at (5.75,0.4) {};

\node[stt] (22) at (6.7,1.2) {};
\node[stt] (23) at (8,1.2) {};
\node[stt] (24) at (9.3,1.2) {};
\node[stt] (25) at (6.40,0.4) {};
\node[stt] (26) at (7,0.4) {};
\node[stt] (27) at (7.70,0.4) {};
\node[stt] (28) at (8.30,0.4) {};
\node[stt] (29) at (9,0.4) {};
\node[stt] (30) at (9.6,0.4) {};
\node[stt] (31) at (6.25,-0.4) {};
\node[stt] (32) at (6.55,-0.4) {};
\node[stt] (33) at (6.85,-0.4) {};
\node[stt] (34) at (7.15,-0.4) {};
\node[stt] (35) at (7.55,-0.4) {};
\node[stt] (36) at (7.85,-0.4) {};
\node[stt] (37) at (8.15,-0.4) {};
\node[stt] (38) at (8.45,-0.4) {};
\node[stt] (39) at (8.85,-0.4) {};
\node[stt] (40) at (9.15,-0.4) {};
\node[stt] (41) at (9.45,-0.4) {};
\node[stt] (42) at (9.75,-0.4) {};

\node[stt] (43) at (10.4,1.2) {};
\node[stt] (44) at (11,1.2) {};
\node[stt] (45) at (11.6,1.2) {};
\node[stt] (46) at (10.25,0.4) {};
\node[stt] (47) at (10.55,0.4) {};
\node[stt] (48) at (10.85,0.4) {};
\node[stt] (49) at (11.15,0.4) {};
\node[stt] (50) at (11.45,0.4) {};
\node[stt] (51) at (11.75,0.4) {};

\node[stt] (52) at (12.5,1.2) {};
\node[stt] (53) at (13,1.2) {};
\node[stt] (54) at (13.5,1.2) {};

\draw [arr] (1) -- (3);
\draw [arr] (2) -- (3);
\draw [arr] (3) -- (4);
\draw [arr] (4) -- (5);
\draw [arr] (5) -- (6);
\draw [arr] (6) -- (7);
\draw [arr] (7) -- (8);
\draw [arr] (7) -- (9);

\draw [edg] (10)--(3)--(11);
\draw [edg] (3)-- (12);

\draw [edg] (16)--(13)--(4)--(14)--(18);
\draw [edg] (13)-- (17);
\draw [edg] (14)-- (19);
\draw [edg] (4)--(15)--(20);
\draw [edg] (15)-- (21);

\draw [edg] (31)--(25)--(22)--(5)--(23)--(28)--(38);
\draw [edg] (22)--(26)--(34);
\draw [edg] (25)-- (32);
\draw [edg] (26)-- (33);
\draw [edg] (23)--(27)--(35);
\draw [edg] (27)-- (36);
\draw [edg] (28)-- (37);
\draw [edg] (5)--(24)--(30)--(42);
\draw [edg] (24)--(29)--(39);
\draw [edg] (29)-- (40);
\draw [edg] (30)-- (41);

\draw [edg] (46)--(43)--(6)--(44)--(49);
\draw [edg] (43)-- (47);
\draw [edg] (44)-- (48);
\draw [edg] (6)--(45)--(51);
\draw [edg] (45)-- (50);

\draw [edg] (52)--(7)--(53);
\draw [edg] (7)-- (54);

\end{scope}

\end{tikzpicture}

\caption{EB trees for $(r,z,g)=(3,z,6)$ and
  $(r,z,g)=(3,z,7)$.}\label{fig:EB}
\end{figure}
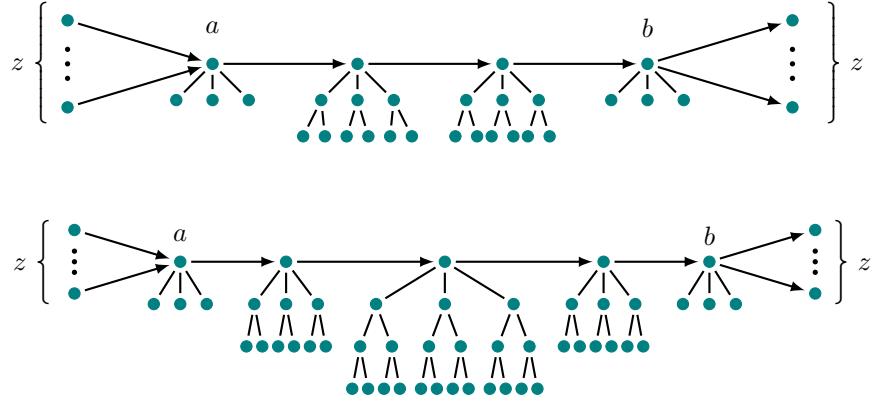 

\section{An upper bound for $(1,z,g)$-mixed cages}\label{sec:upper}

In \cite{AHM19} the authors gave a pair of circulants which are
$(1,z,4)$-mixed graphs and provide the upper bounds in [{\sc ahm6}]
and [{\sc ahm7}]. They also proved that these upper bounds are sharp
for $z=1$ and $z=2$. However, our computational results show that
these bounds are attained at least for $z\leq 8$. Moreover, the same
idea can be generalized to all $g\geq 3$ and it still matches the order of
all of our computed mixed cages for $r=1$. The same pattern of jumps
is used and only the order of the circulant required
amendment. 

\begin{defin}
  Let $z\geq 1$ and $g\geq 3$. Define 
$$N(z,g)=\begin{cases}
z(g-1)+2 & \textup{$z$ even,}\\
(z+1)(g-1) & \textup{$z$ odd.}\\
\end{cases}$$
Now, take $m=\frac{N(z,g)}{2}$ (which is an integer) and define the
following circulant:
$$M(z,g) = \overrightarrow{C}_{2m}
\left(1,2,\ldots,\left\lceil \frac{z}{2}\right\rceil;\pm m;
m+1,m+2,\ldots,m+\left\lfloor \frac{z}{2} \right\rfloor\right).$$
\end{defin}

In the previous definition, the semicolons are there just for
readability, separating three distinct patterns in the jumps. Also,
note that $m \equiv -m \mod 2m$, and hence this jump would produce
pairs of antiparallel arcs, which are forbidden; by our convention,
they must be understood as edges connecting antipodal vertices. The
``$\pm$'' sign is there just for emphasis.

\begin{theorem}\label{thm:MCirc}
$M(z,g)$ is a $(1,z,g)$-mixed graph, hence $n(1,z,g)\leq N(z,g)$.  
\end{theorem}
\begin{proof}
  It is clear that $M(z,g)$ is 1-regular in edges and $z$-regular in
  arcs. It remains to be proved that its girth is exactly $g$. Take
  $n=N(z,g)$, $m=\frac{n}{2}$ and
  $h = \left\lceil\frac{z}{2}\right\rceil$.

  Assume first that $z$ is even. Since $n-(g-1)h = (g-1)h+2=m+1$, we
  have that $(g-1)h\ra n$ is an arc and the cycle
  $0\ra h \ra 2h \ra \cdots \ra (g-1)h\ra n = 0$ proves that
  $g(M(z,g))\leq g$. Now assume there is a shorter cycle
  $(x_{0},x_{1},\ldots,x_{\ell})$ with $x_{\ell}= x_{0}$ and
  $\ell \leq g-1$. The differences $j_{k}=x_{k+1}-x_{k}$ must belong
  to the set of jumps and must satisfy
  $\sum_{k=0}^{\ell-1} j_{k}\equiv 0 \mod n$. Now, the jumps $j_{k}$
  are either small jumps $j_{k}=s_{k}$ for $s_{k}\in\{1,2,\ldots,h\}$
  or big jumps $j_{k}=m+s_{k}$ for $s_{k}\in\{0,1,\ldots,h\}$. Assume
  there are a total of $b$ big jumps, then
  $\sum_{k=0}^{\ell-1} j_{k} =bm + \alpha$ where
  $\alpha =\sum_{k=0}^{\ell-1} s_{k}$. Now, we have
  $\alpha \leq h\ell \leq h(g-1)=m-1$. Since no two consecutive jumps
  could be both equal to the antipodal jump $\pm m$, we also have
  $\alpha \geq \ell/2\geq 1$. It follows that
  $\sum_{k=0}^{\ell-1} j_{k} = bm+ \alpha$ for some
  $\alpha\in \{1,2, \ldots, m-1\}$. Therefore
  $\sum_{k=0}^{\ell-1} j_{k} = bm+ \alpha \not\equiv 0 \mod n$, a
  contradiction.
 
  Now assume that $z$ is odd. In this case, $h=\frac{z+1}{2}$. Since
  $n-(g-1)h = \frac{n}{2}$, we have that $(g-1)h\ed n$ is an edge and
  $0\ra h \ra 2h \ra \cdots \ra (g-1)h\ed n = 0$ is a cycle that proves
  $g(M(z,g))\leq g$. Now, assume there is a shorter cycle
  $(x_{0},x_{1},\ldots,x_{\ell})$ with $x_{\ell}= x_{0}$ and
  $\ell \leq g-1$. The differences $j_{k}=x_{k+1}-x_{k}$ must belong
  to the set of jumps and must satisfy
  $\sum_{k=0}^{\ell-1} j_{k}\equiv 0 \mod n$. As before,
  $\sum_{k=0}^{\ell-1} j_{k} = bm+ \alpha$, but this time
  $\alpha \leq h\ell\leq h(g-1) = m$. However the exact value
  $\alpha = m$ can only be achieved when
  $j_{k} =h = \left\lceil \frac{z}{2} \right\rceil$ for all $k$ and in
  this case we would have $b=0$. Therefore we have either
  $\sum_{k=0}^{\ell-1} j_{k} = \alpha$ with
  $\alpha\in \{1,2,\ldots,m\}$ or
  $\sum_{k=0}^{\ell-1} j_{k} = bm+ \alpha$ with
  $\alpha\in \{1,2,\ldots,m-1\}$, which in both cases imply
  $\sum_{k=0}^{\ell-1} j_{k}\not\equiv 0 \mod n$. A contradiction.
\end{proof}

The above theorem (see also [{\sc dp2}] and [{\sc dp3}]) is a
generalization of previous partial results [{\sc ahm6}], [{\sc ahm7}],
[{\sc e1}] and [{\sc e2}], as well as an improvement on the bounds in
[{\sc jj2}] and [{\sc jj3}].  Recall from Section~\ref{sec:prelim}
that the circulant
$\BCW(z,g) = \overrightarrow{C}_{n}(1,2,\ldots,z)$ with $n=z(g-1)+1$ is
a $(0,z,g)$-mixed graph (a digraph, actually), which was conjectured
in \cite{BCW70} to be minimal (Conjecture~\ref{conj:BCW}).  We point
out that in case $z$ is even, the
construction above requires only one additional vertex
($|M(z,g)|=|\BCW(z,g)|+1$) to add all the $n/2$ required
edges. Moreover, when $z$ is even, this construction coincides with
the following composition $M(z,g)=\BCW(\frac{z}{2},g)[K_{2}]$; this
fact was noted by Exoo in the case $z=2$ \cite{Exo23} (see also [{\sc
  e2}]). We already mentioned that all our computational results
suggest that this bound is sharp.  Considering all of this, we propose
the following:

\begin{conjecture}\label{conj:1zg}
$n(1,z,g) = N(z,g)=\begin{cases}
z(g-1)+2 & \textup{$z$ even,}\\
(z+1)(g-1) & \textup{$z$ odd.}\\
\end{cases}$  
\end{conjecture}

\section{Mixed cages of girth three}\label{sec:g3}

The extended bound for $g=3$ is $n(r,z,3)\geq r+2z+1$, but actually, in
this case, the equality holds:

\begin{theorem}\label{thm:g3}
$n(r,z,3)=r+2z+1$.
\end{theorem}
\begin{proof}
  By the extended bound, $n(r,z,3)\geq r+2z+1$. Let $G$ be a mixed
  graph of order $r+2z+1$, with vertex set
  $V=\{i : i\in \mathbb{Z}_{2z+r+1}\}$ and the following adjacencies:

$$ \begin{array}{lll}
N^{+}(i) &=& \{i+\alpha: 1\leq \alpha \leq z\},\\
N(i) &=& \{i+\alpha: z+1 \leq \alpha \leq r+z\},\\
N^{-}(i) &=& \{i+\alpha: r+z+1\leq \alpha \leq r+2z\}.
\end{array}$$

It should be clear that $(0,z,z+1,0)$ is
always a 3-cycle, and hence $G$ is an $(r,z,3)$-mixed graph.
\end{proof}

\section{Mixed cages of girth four}\label{sec:g4}

Recall that $\EB(r,z,4) =2r+2z+2 \leq n(r,z,4)$, and by
Theorem~\ref{JJ_bound} we have that $2r+2z+2\leq n(r,z,4) \leq 2r+4z$
(compare with [{\sc ahm2}] and [{\sc jj1}]). In this section we shall
prove Theorem~\ref{thm:2z-2} which states that for $(r,z)\neq (1,2)$,
we get:
\begin{align*}
r\leq 2z-2\implies &n(r,z,4)\geq 3r+2z+2,\\
r\geq 2z-2\implies &n(r,z,4) = 2r+4z. 
\end{align*}

We note that $3r+2z+2 \geq 2r+4z$ precisely when $r\geq 2z-2$ and that
both bounds coincide exactly when $r=2z-2$. Comparing with our
computational results in Table~\ref{tab:g=4}, we also point out that
$n(2,3,4)=16$ and $n(3,3,4)=18$ do attain the value $2r+4z$, although
these cases are not covered by Theorem~\ref{thm:2z-2}. On the other
hand, also comparing with Table~\ref{tab:g=4}, the cases $r=1$ with
$z\geq 2$, as well as $n(2,4,4)=16$ and $n(2,5,4)=20$, confirm that
the upper bound $2r+4z$ is not attained in general when $r< 2z-2$.

For the rest of the section we shall consider an $(r,z,4)$-mixed graph
$G_{0}$ that does not attain the upper bound and hence
$2r+2z+2\leq |G_{0}| < 2r+4z$.  Given an arc $a\ra b$ in $G_{0}$, we
can construct an EB tree in $G_{0}$ starting with it, and then we get
a subgraph of $G_{0}$ as in Figure~\ref{fig:G0}. The $m$ extra vertices
$\{t_{1},t_{2},\ldots, t_{m}\}$ are the vertices of $G_{0}$ that do
not belong to the EB tree. For future reference, we define
$A=N^{-}(a)$, $B=N^{+}(b)$, $P=N(a)$, $Q=N(b)$ and
$T=\{t_{1},t_{2}, \ldots, t_{m}\}$. We summarize our assumptions as
follows:
\begin{align*}
&r\geq 1, z\geq 1, \\
&0\leq m < 2z-2,\\
&|G_{0}| = 2r+2z+2+m, \\
&G_{0}\textup{ is an }(r,z,4)\textup{-mixed graph}.
\end{align*}

By the girth condition, we know that all $3$-cycles are forbidden in
$G_{0}$. Moreover, the next two lemmas show that the non-cyclic
triangles are also forbidden when $r$ is large enough (see
Figure~\ref{fig:3cycles} and Figure~\ref{fig:untriangles}).

\begin{figure}[!ht]
\begin{center}
\begin{tikzpicture}[scale = 1]

\node[stw] (1) at (2,2) {$a$};
\node[stw] (2) at (5,2) {$b$};

\node[stw] (3) at ( -0.60, 3.5 ) {$a_{z}$};
\node[stw] (4) at ( -1., 2. ) {$a_{1}$};
\node[stw] (5) at ( -0.60, 0.5 ) {$p_{r}$};
\node[stw] (6) at ( 0.5, -0.60
 ) {$p_{1}$};
\node[stw] (7) at ( 2.6, -1.4 ) {$t_{1}$};
\node[stw] (8) at ( 4.4, -1.4 ) {$t_{m}$};
\node[stw] (9) at ( 6.5, -0.6 ) {$q_{1}$};
\node[stw] (10) at ( 7.6, 0.5 ) {$q_{r}$};
\node[stw] (11) at ( 8., 2. ) {$b_{1}$};
\node[stw] (12) at ( 7.6, 3.5 ) {$b_{z}$};

\node[dot] at (-0.95,2.52) {};
\node[dot] at (-0.9,2.78) {};
\node[dot] at (-0.82,3.03) {};

\node[dot] at (-0.3,0.07) {};
\node[dot] at (-0.12,-0.12) {};
\node[dot] at (0.07,-0.3) {};

\node[dot] at (3.22,-1.4) {};
\node[dot] at (3.5,-1.4) {};
\node[dot] at (3.78,-1.4) {};

\node[dot] at (6.93,-0.3) {};
\node[dot] at (7.12,-0.12) {};
\node[dot] at (07.3,0.07) {};

\node[dot] at (7.95,2.52) {};
\node[dot] at (7.9,2.78) {};
\node[dot] at (7.82,3.03) {};

\draw[edg] (5)--(1)--(6);
\draw[edg] (9)--(2)--(10);

\draw [arr] (1) -- (2);
\draw [arr] (3) -- (1);
\draw [arr] (4) -- (1);
\draw [arr] (2) -- (11);
\draw [arr] (2) -- (12);

\node[rotate=-15] at ($(2,2)+(165:3.3)$) {$\left\{\vphantom{\rotatebox{90}{mmmm}}\right.$};
\node at ($(2,2)+(166:3.7)$) {$A$};
\node[rotate=45] at ($(2,2)+(224.5:3.3)$) {$\left\{\vphantom{\rotatebox{90}{mmmm}}\right.$};
\node at ($(2,2)+(224.5:3.65)$) {$P$};
\node[rotate=90] at ($(3.5,2)+(270:3.75)$) {$\left\{\vphantom{\rotatebox{90}{mmmme}}\right.$};
\node at ($(3.5,2)+(270:4.1)$) {$T$};
\node[rotate=-163] at ($(5,2)+(15:3.3)$) {$\left\{\vphantom{\rotatebox{90}{mmmm}}\right.$};
\node at ($(5,2)+(15:3.7)$) {$B$};
\node[rotate=135] at ($(5,2)+(-45.5:3.3)$) {$\left\{\vphantom{\rotatebox{90}{mmmm}}\right.$};
\node at ($(5,2)+(-45.5:3.65)$) {$Q$};
\end{tikzpicture}
\end{center}
\caption{A subgraph of graph $G_{0}$.}
\label{fig:G0}
\end{figure}
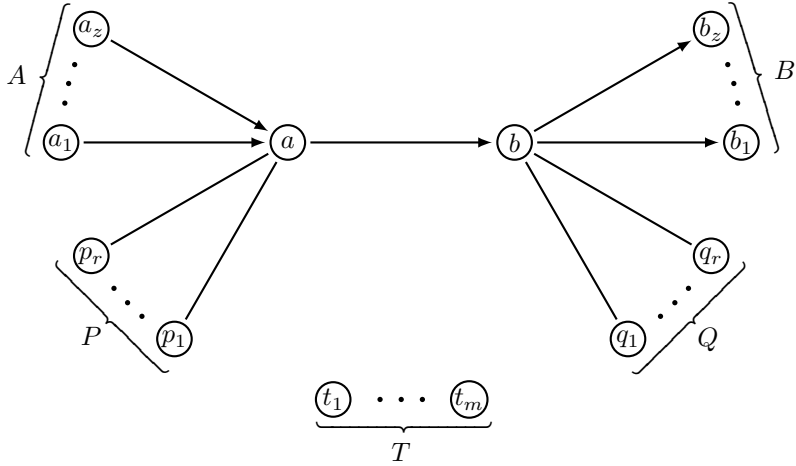

\begin{lemma}\label{lem:NoTT3}
If $r\geq m$, then \utA is forbidden in $G_{0}$.
\end{lemma}
\begin{proof}
  By way of contradiction, assume $r\geq m$ and take a non-cyclic
  triangle \utA in $G_{0}$. 
  Construct an EB tree in $G_{0}$ starting with $a\ra b$, as in
  Figure~\ref{fig:G0}. Clearly $N^{+}(a)\cap N^{-}(b)\subseteq T$ and
  hence $w \in T$. But now, the girth condition gives us
  $N(w)\cap X = \varnothing$ for every $X=A,P,Q,B$. Also,
  $a,b\not\in N(w)$ since, by construction, $T$ is disjoint from
  $N(a)=P$ and $N(b)=Q$. It follows that
  $N(w)\subseteq T\setminus \{w\}$. But
  $|N(w)|=r>m-1= |T\setminus \{w\}|$. A contradiction.
\end{proof}

\begin{lemma}
  \label{lem:m+1}
  If $r\geq m+1$ and $(r,z)\neq (1,2)$ then \utB and
  \utC are forbidden in $G_{0}$.
\end{lemma}
\begin{proof}
  We shall put some expressions in bold to emphasize partial results
  that will be used later within this proof. Since both kinds of
  non-cyclic triangles are dual to each other, we may consider just
  the first one. By way of contradiction, assume $r\geq m+1$ and take
  a non-cyclic triangle \utB in $G_{0}$.  The existence of such a
  triangle readily implies that $\bm{z\geq 2}$. Construct an EB tree
  in $G_{0}$ starting with $a\ra b$, as in Figure~\ref{fig:G0}. It
  follows that $w \in P$. We find it convenient to further split
  vertices in $P$ and $Q$ so that $Q=Q_{1}\cupdot Q_{2}$ and
  $P=P_{1}\cupdot P_{2}$ as in Figure~\ref{fig:G1} (solid part). The
  exact definitions of these decompositions will be given in items~3
  and~5 below.  In order to give full details in a compact manner we
  shall use expressions like
  \begin{center}
    ``$Y\cap X = \varnothing$ for $X=U, V, W$ due to $(d,g,l)$''
  \end{center} 
  to mean ``a vertex in $Y\cap U$ would violate the \textbf{d}isjoint
  condition; a vertex in $Y\cap V$ would violate the \textbf{g}irth
  condition; and a vertex in $Y\cap W$ would violate
  \textbf{L}emma~\ref{lem:NoTT3}''. The numbering of the following
  items matches that of the dashed arcs and edges in
  Figures~\ref{fig:G1} and~\ref{fig:G2}.


 \begin{enumerate}
 \item $\bm{w{\rab} b}$ and $\bm{w \in P}$: Since $a\ed w \ra b$.

 \item $\bm{w{\edb} T}$: $N(w)\cap X=\varnothing$ for
   $X=A,P,Q,B,\{b\}$ due to $(g,g,g,g,d)$. Hence
   $N(w)\subseteq T\cup\{a\}$, but this is only possible when
   $r\leq m+1$ and hence $\bm{r=m+1}$ (otherwise we get a
   contradiction here). Then we have $N(w)=T\cup \{a\}$. Thus
   $w\ed T$.

 \item $\bm{a{\rab} Q_{1}}$: $N^{+}(a)\cap X=\varnothing$ for
   $X=A,P,T,B$ due to $(d,d,g,l)$. Hence
   $N^{+}(a)\subseteq Q\cup\{b\}$, but this is only possible when
   $\bm{z\leq r+1}$ (otherwise, we get a contradiction here). Now,
   define $Q_{1}=N^{+}(a)\setminus\{b\}$ and $Q_{2}=Q\setminus Q_{1}$,
   note that $|Q_{1}|=z-1\geq 1$ and $|Q_{2}|=r-z+1\geq 0$. Thus we
   have $a\ra Q_{1}$.
   \end{enumerate}

   \noindent At this point, we know that $z\geq 2$, $r=m+1$ and
   $z\leq r+1$. It follows that $m\neq 0$ since otherwise we would
   have $(r,z)=(1,2)$, which is explicitly excluded in the statement
   of this lemma. It follows that $\bm{m\geq 1}$ and, in particular, that
   $T\neq \varnothing$. This will be required in item~9.

\begin{enumerate}[start=4]
 \item $\bm{Q_{1}{\edb} T}$: Given $q\in Q_{1}$ we get
   $N(q)\cap X=\varnothing$ for $X=A,P,Q,B,\{a\}$ due to
   $(g,g,g,g,d)$. Hence $N(q)\subseteq T\cup\{b\}$ and, since $r=m+1$,
   $N(q)= T\cup\{b\}$. Thus $Q_{1}\ed T$.
   \end{enumerate}

   \noindent Items~5 and~6 will now supersede items~1 and~2, as shown
   in Figure~\ref{fig:G2}.

   \begin{enumerate}[start=5]
   \item $\bm{b {\lab} P_{1}}$: $N^{-}(b)\cap X=\varnothing$ for
     $X=A,T,Q,B$ due to $(l,g,d,d)$. Hence
     $N^{-}(b)\subseteq P\cup\{a\}$. Now define
     $P_{1}=N^{-}(b)\setminus\{a\}$ and $P_{2}=P\setminus P_{1}$, note
     that $|P_{1}|=z-1\geq 1$ and $|P_{2}|=r-z+1\geq 0$. Thus we have
     $b\la P_{1}$.

   \item $\bm{P_{1}{\edb} T}$: Given $p\in P_{1}$ we get
     $N(p)\cap X=\varnothing$ for $X=A,P,Q,B,\{b\}$ due to
     $(g,g,g,g,d)$. Hence $N(p)\subseteq T\cup\{a\}$ and, since
     $r=m+1$, $N(p)= T\cup\{a\}$. Thus $P_{1}\ed T$.

   \item $\bm{P_{1}{\lab} A}$: Given $p\in P_{1}$ we get
     $N^{-}(p)\cap X=\varnothing$ for $X=P,T,Q,B,\{a\},\{b\}$ due to
     $(g,d,g,g,d,d)$. Hence $N^{-}(p)= A$. Thus $P_{1}\la A$.

   \item $\bm{Q_{1}{\rab} B}$: Given $q\in Q_{1}$ we get
     $N^{+}(q)\cap X=\varnothing$ for $X=A,P,T,Q,\{a\},\{b\}$ due
     to $(g,g,d,g,d,d)$. Hence $N^{+}(q)= B$. Thus $Q_{1}\ra B$.

   \item $\bm{P_{1}{\rab} Q_{2}}$: Given $p\in P_{1}$ we get
     $N^{+}(p)\cap X=\varnothing$ for $X=A,P,T,Q_{1},B,\{a\}$ due to
     $(d,g,d,g,l,d)$. Hence $N^{+}(p)\subseteq Q_{2}\cup\{b\}$. But
     this is only possible when $z\leq (r-z+1)+1$, that is, when
     $r\geq 2z-2$. Since we also know that $r=m+1 \leq 2z-2$, we get
     $r=2z-2$. It follows that $\bm{|Q_{2}|= r-z+1 = z-1\geq 1}$ and
     $N^{+}(p)= Q_{2}\cup\{b\}$. Thus $P_{1}\ra Q_{2}$.
 \end{enumerate}

 Note that for $b'\in B$ we already have $N^{-}(b') =Q_{1}\cup \{b\}$,
 hence the vertices of $B$ already have their in-neighborhood
 completed and cannot receive any more in-neighbors. In particular,
 if we take $q\in Q_{2}$, we have $N^{+}(q)\cap B =
 \varnothing$. Besides that, we also have $N^{+}(q)\cap X=\varnothing$
 for $X=A,P_{1},T,Q,\{a\},\{b\}$ due to $(g,d,g,g,g,d)$. Hence
 $N^{+}(q)\subseteq P_{2}$, but this is only possible when
 $z\leq r-z+1 = z-1$ (see item~9). A contradiction.
\end{proof}

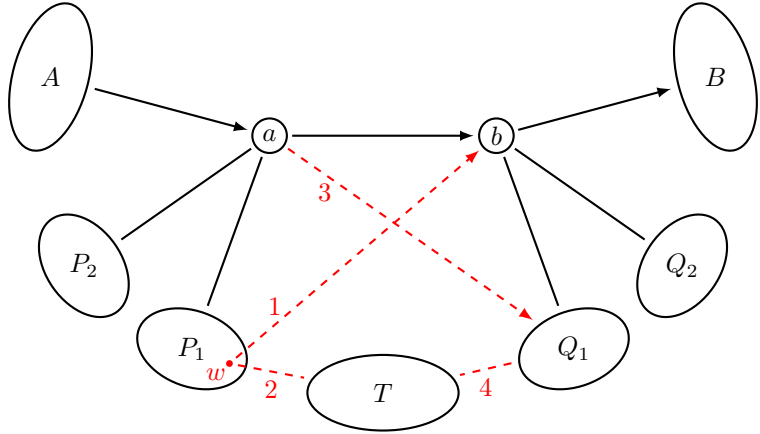
\begin{figure}[!ht]
\begin{center}
\begin{tikzpicture}[scale = 1]

\node[stw] (a) at (2,2) {$a$};
\node[stw] (b) at (5,2) {$b$};

\node[ellip1,rotate=75] (A) at ($(2,2)+(165:3cm)$) {};
\node[ellip2,rotate=125] (P2) at ($(2,2)+(215:3cm)$) {};
\node[ellip2,rotate=160] (P1) at ($(2,2)+(250:3cm)$) {};
\node[ellip1,rotate=0] (T) at ($(3.5,2)+(270:3.4cm)$) {};
\node[ellip1,rotate=105] (B) at ($(5,2)+(15:3cm)$) {};
\node[ellip2,rotate=55] (Q2) at ($(5,2)+(-35:3cm)$) {};
\node[ellip2,rotate=20] (Q1) at ($(5,2)+(-70:3cm)$) {};
\node at ($(2,2)+(165:3cm)$) {$A$};
\node at ($(2,2)+(215:3cm)$) {$P_{2}$};
\node at ($(2,2)+(250:3cm)$) {$P_{1}$};
\node at ($(3.5,2)+(270:3.4cm)$) {$T$};
\node at ($(5,2)+(15:3cm)$) {$B$};
\node at ($(5,2)+(-35:3cm)$) {$Q_{2}$};
\node at ($(5,2)+(-70:3cm)$) {$Q_{1}$};
\node[dot2,red] (p1) at ($(2,2)+(260:3.06cm)$) {};
\node[red] at ($(2,2)+(257.4:3.24cm)$) {$w$};

\draw[edg] (P2)--(a)--(P1);
\draw[edg] (Q1)--(b)--(Q2);

\draw [arr] (a) -- (b);
\draw [arr] (A) -- (a);
\draw [arr] (b) -- (B);
\draw [arr,red,dashed] (p1)--node[above, pos=0.16]{1} (b);
\draw [edg,red,dashed] (p1)--node[below]{2} (T);
\draw [arr,red,dashed] (a)--node[below, pos=0.15]{3} (Q1);
\draw [edg,red,dashed] (Q1)--node[below]{4} (T);
\end{tikzpicture}
\end{center}
\caption{Graph as in items 1--4 of the proof of
  Lemma~\ref{lem:m+1}.}
\label{fig:G1}
\end{figure}

\begin{figure}[!ht]
\begin{center}
\begin{tikzpicture}[scale = 1]

\node[stw] (a) at (2,2) {$a$};
\node[stw] (b) at (5,2) {$b$};

\node[ellip1,rotate=75] (A) at ($(2,2)+(165:3cm)$) {};
\node[ellip2,rotate=125] (P2) at ($(2,2)+(215:3cm)$) {};
\node[ellip2,rotate=160] (P1) at ($(2,2)+(250:3cm)$) {};
\node[ellip1,rotate=0] (T) at ($(3.5,2)+(270:3.4cm)$) {};
\node[ellip1,rotate=105] (B) at ($(5,2)+(15:3cm)$) {};
\node[ellip2,rotate=55] (Q2) at ($(5,2)+(-35:3cm)$) {};
\node[ellip2,rotate=20] (Q1) at ($(5,2)+(-70:3cm)$) {};
\node at ($(2,2)+(165:3cm)$) {$A$};
\node at ($(2,2)+(215:3cm)$) {$P_{2}$};
\node at ($(2,2)+(250:3cm)$) {$P_{1}$};
\node at ($(3.5,2)+(270:3.4cm)$) {$T$};
\node at ($(5,2)+(15:3cm)$) {$B$};
\node at ($(5,2)+(-35:3cm)$) {$Q_{2}$};
\node at ($(5,2)+(-70:3cm)$) {$Q_{1}$};

\draw[edg] (P2)--(a)--(P1);
\draw[edg] (Q1)--(b)--(Q2);

\draw [arr] (a) -- (b);
\draw [arr] (A) -- (a);
\draw [arr] (b) -- (B);
\draw [arr,red,dashed] (P1)--node[above, pos=0.3]{5} (b);
\draw [edg,red,dashed] (P1)--node[below]{6} (T);
\draw [arr] (a)--(Q1);
\draw [edg] (Q1)--(T);
\draw [arr,red,dashed] (A) to[bend left = 15]node[above, pos=0.4]{7} (P1);
\draw [arr,red,dashed] (Q1) to[bend left = 15]node[above, pos=0.6]{8} (B);
\draw [arr,red,dashed] (P1) to[bend left = 15]node[below, pos=0.4]{9} (Q2);
\end{tikzpicture}
\end{center}
\caption{Graph as in items 5--9 of the proof of
  Lemma~\ref{lem:m+1}.}
\label{fig:G2}
\end{figure}
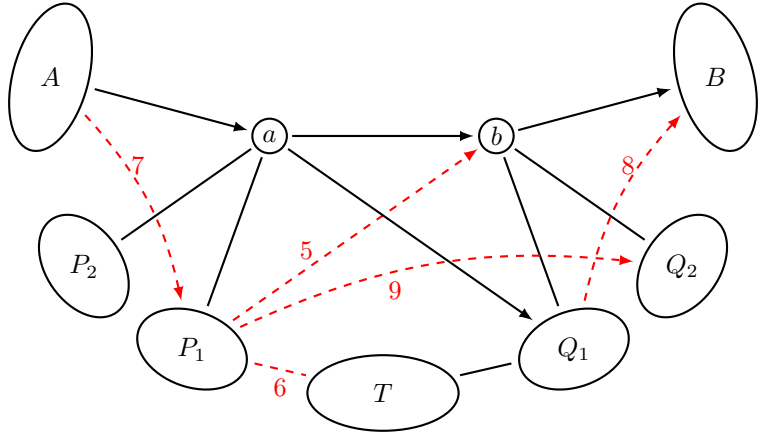

Now we can prove the main theorem of this section.

\begin{theorem}\label{thm:2z-2}
We have $n(1,2,4)= 8$. Also, for $(r,z)\neq (1,2)$, we get:
\begin{align*}
r\leq 2z-2\implies &n(r,z,4)\geq 3r+2z+2,\\
r\geq 2z-2\implies &n(r,z,4) = 2r+4z. 
\end{align*}
\end{theorem}
\begin{proof}
  $n(1,2,4)= 8$ follows from Theorem~\ref{thm:EB} and [{\sc ahm7}].

  Assume $(r,z)\neq (1,2)$ and take $m\geq 0$ such that
  $m<\min(r,2z-2)$. Then we have $r\geq m+1$ and $m<2z-2$. Hence, both
  Lemmas~\ref{lem:NoTT3} and~\ref{lem:m+1} apply here and all
  non-cyclic triangles are forbidden in any $(r,z,4)$-mixed graph $G$
  of order $|G| = 2r+2z+2+m <2r+4z$. But then, given any arc $a\ra b$,
  all these sets must be pairwise disjoint:
  $N(a),N^{+}(a),N^{-}(a),N(b),N^{+}(b),N^{-}(b)$, and the union of
  those must contain $2r+4z$ vertices. Hence $|G|\geq 2r+4z$. A
  contradiction.

  It follows that any $(r,z,4)$-mixed graph $G$ of order $|G| = 2r+2z+2+m$
  must have $m\geq \min(r,2z-2)$. When $r\leq 2z-2$ this means
  $n(r,z,4)\geq 3r+2z+2$ as claimed. Otherwise we have $r\geq 2z-2$,
  and hence $n(r,z,4)\geq 2r+4z$, which together with
  Theorem~\ref{JJ_bound} gives us $n(r,z,4)=2r+4z$.
\end{proof}











\section{Computational results}\label{sec:exp}

Tables~\ref{tab:g=3}--\ref{tab:g=10} present our computational
results. We used exhaustive backtracking searches with symmetry
reductions, as explained in detail in \cite{DP26}. Our code can be
consulted in \cite{MTCAGES11}, which is an adaptation of the code we
developed for (normal) cages \cite{CAGES10}. Our algorithms ran on an
Intel Core i7-10700K at 3.8~GHz, always using a single core. We
granted one hour of running time for each case, and some cases were
granted additional time (up to 39.3 days) when they seemed
promising. We also found out that it is often the case that when one
mixed cage exists, many more (up to thousands) also exist. Hence we
also put a limit of 20 mixed cages per case and the computation was
ended when this limit was reached. Our code was written in
GAP~\cite{GAP414} using YAGS~\cite{YAGS6} and
\texttt{nauty}~\cite{MP14,MP21}.

In the tables, an entry like \cell{N}{e}{s} means that there are
exactly $s$ mixed cages of order $N$, with an excess of $e$ with
respect to the extended bound $\EB(r,z,g)$ in Theorem~\ref{thm:EB}
($e=N-\EB(r,z,g)$). The subscript $e$ is omitted when its value is
zero. An entry like \cell{N}{e}{s\smp} means that we found $s$ mixed
cages but the search space was not exhausted in the allotted time and
hence additional mixed cages may exist. An entry like \cell{N}{e}{+}
means that our algorithms did not find any mixed cages but the search
space was exhausted up to $N-1$ so that any mixed cage with the
corresponding parameters must have order at least $N$. Also, in
Table~\ref{tab:g=4}, we shaded the cells where the condition
$r\geq 2z-2$ is met and hence the exact values of $n(r,z,4)$ are known
from Theorem~\ref{thm:2z-2}. Moreover, we write $N$ in boldface
whenever we know that a mixed cage of that order exists and hence the
value of $N=n(r,z,g)$ is exact: in most cases this is because our
algorithms did find at least one mixed cage, but in six special cases
the reason is that our algorithmic lower bound matched some previous
upper bound. These six cases are: $n(1,7,4)=24$ [{\sc ahm6}];
$n(11,z,4)=22+4z$ for $z\in\{2,3,4,5\}$ [{\sc jj1}]; and $n(2,4,5)=25$
[{\sc adg3}], [{\sc e3}]. Note that there are four values in the
shaded area of Table~\ref{tab:g=4} where our algorithms could not achieve
the exact values known from our Theorem~\ref{thm:2z-2}.

All of our computational results were verified to be compatible with
all our theorems and the known results in the literature
\cite{AHM19,ADG22,Exo23,JJ25,AM26,AM26a, JJKP25} (see
Tables~\ref{tab:ref1} and~\ref{tab:ref2}). Table~\ref{tab:g=3} is only
a fraction of what can be computed with our algorithms, since the full
table would be too big to reproduce here (for instance, $n(6,10,3)=27$
was computed in only 25.8 seconds).

Among the many new mixed cages and lower bounds reported in
Tables~\ref{tab:g=3} to~\ref{tab:g=10}, we were able to confirm that
there are exactly 23 $(3,1,5)$-mixed cages of order 24 as reported by
Exoo in \cite{Exo23}, but we also improved all six lower bounds
(compare [{\sc e10}]--[{\sc e15}] with [{\sc dp8}]--[{\sc dp13}]) and
determined the exact value of $n(r,z,g)$ for two of those cases,
namely $n(4,1,6)=48$ and $n(3,1,8)=76$. Of course, computational
techniques are unable to reproduce theoretical results that apply for
an infinite number of parameter triples, like those in [{\sc ahm2}],
[{\sc ahm4}], [{\sc ahm5}], [{\sc e1}] and [{\sc e2}]; but besides
that, our algorithms were also able to reproduce almost all previously
known exact values [{\sc e6}]--[{\sc e9}], with the only exception
being [{\sc adg4}], which was obtained by theoretical arguments.

We point out that $n(r,z,g)$ is not strictly monotone in $z$, as
$n(2,3,4)=n(2,4,4)=16$ and $n(2,4,5)=n(2,5,5)=25$.  Also, for $g=3$
and $g=4$, we find it interesting that in most cases where we were
able to compute a mixed cage, the limit of 20 mixed cages was reached.

\setlength{\tabcolsep}{3.5pt}

\begin{table}[!ht]
\begin{tabular}{|>{\columncolor{azulito}}r||r|r|r|r|r|r|r|r|r|r|}
\hline 
\rowcolor{azulito}
$r\backslash z$ &1 &2 &3 &4 &5 &6 &7 &8 &9 &10\\ \hline\hline
1 & \cellb{4}{}{1}  & \cellb{6}{}{4}  & \cellb{8}{}{20\smp}  & \cellb{10}{}{20\smp}  & \cellb{12}{}{20\smp}  & \cellb{14}{}{20\smp}  & \cellb{16}{}{20\smp}  & \cellb{18}{}{20\smp}  & \cellb{20}{}{20\smp}  & \cellb{22}{}{20\smp} \\\hline
2 & \cellb{5}{}{1}  & \cellb{7}{}{9}  & \cellb{9}{}{20\smp}  & \cellb{11}{}{20\smp}  & \cellb{13}{}{20\smp}  & \cellb{15}{}{20\smp}  & \cellb{17}{}{20\smp}  & \cellb{19}{}{20\smp}  & \cellb{21}{}{20\smp}  & \cellb{23}{}{20\smp} \\\hline
3 & \cellb{6}{}{2}  & \cellb{8}{}{20\smp}  & \cellb{10}{}{20\smp}  & \cellb{12}{}{20\smp}  & \cellb{14}{}{20\smp}  & \cellb{16}{}{20\smp}  & \cellb{18}{}{20\smp}  & \cellb{20}{}{20\smp}  &
 \cellb{22}{}{20\smp}  & \cellb{24}{}{20\smp} \\\hline
4 & \cellb{7}{}{2}  & \cellb{9}{}{20\smp}  & \cellb{11}{}{20\smp}  & \cellb{13}{}{20\smp}  & \cellb{15}{}{20\smp}  & \cellb{17}{}{20\smp}  & \cellb{19}{}{20\smp}  & \cellb{21}{}{20\smp}  &
 \cellb{23}{}{20\smp}  & \cellb{25}{}{20\smp} \\\hline
5 & \cellb{8}{}{3}  & \cellb{10}{}{20\smp}  & \cellb{12}{}{20\smp}  & \cellb{14}{}{20\smp}  & \cellb{16}{}{20\smp}  & \cellb{18}{}{20\smp}  & \cellb{20}{}{20\smp}  & \cellb{22}{}{20\smp}  
& \cellb{24}{}{20\smp}  & \cellb{26}{}{20\smp} \\\hline
6 & \cellb{9}{}{4}  & \cellb{11}{}{20\smp}  & \cellb{13}{}{20\smp}  & \cellb{15}{}{20\smp}  & \cellb{17}{}{20\smp}  & \cellb{19}{}{20\smp}  & \cellb{21}{}{20\smp}  & \cellb{23}{}{20\smp}  
& \cellb{25}{}{20\smp}  & \cellb{27}{}{20\smp} \\\hline
\end{tabular}
\caption{$n(r,z,3)$.}\label{tab:g=3}
\end{table}

\begin{table}[!ht]
\begin{tabular}{|>{\columncolor{azulito}}r||r|r|r|r|r|r|r|r|r|}
\hline 
\rowcolor{azulito}
$r\backslash z$&1 &2 &3 &4 &5 &6 &7 &8 &9\\ \hline\hline
1 & \cellshbb{6}{}{1}  & \cellbb{8}{}{1}  & \cellbb{12}{2}{1}  & \cellbb{14}{2}{1}  & \cellb{18}{4}{1\smp}  & \cellbb{20}{4}{1}  & \cellb{24}{6}{+}  & \cellb{26}{6}{1\smp}  & \cell{28}{6}{+} \\\hline
2 & \cellshbb{8}{}{2}  & \cellshb{12}{2}{20\smp}  & \cellb{16}{4}{20\smp}  & \cellbb{16}{2}{1}  & \cellb{20}{4}{1\smp}  & \cell{24}{6}{+}  & \cell{25}{5}{+}  & \cell{25}{3}{+}  & \cell{27}{3}{+} \\\hline
3 & \cellshbb{10}{}{2}  & \cellshb{14}{2}{20\smp}  & \cellb{18}{4}{20\smp}  & \cell{20}{4}{+}  & \cell{22}{4}{+}  & \cell{24}{4}{+}  & \cell{26}{4}{+}  & \cell{26}{2}{+}  & \cell{28}{2}{+} \\\hline
4 & \cellshbb{12}{}{4}  & \cellshb{16}{2}{20\smp}  & \cellshb{20}{4}{20\smp}  & \cell{21}{3}{+}  & \cell{23}{3}{+}  & \cell{25}{3}{+}  & \cell{27}{3}{+}  & \cell{29}{3}{+}  & \cell{31}{3}{+} \\\hline
5 & \cellshbb{14}{}{4}  & \cellshb{18}{2}{20\smp}  & \cellshb{22}{4}{20\smp}  & \cell{24}{4}{+}  & \cell{26}{4}{+}  & \cell{28}{4}{+}  & \cell{30}{4}{+}  & \cell{32}{4}{+}  & \cell{34}{4}{+} \\\hline
6 & \cellshbb{16}{}{7}  & \cellshb{20}{2}{20\smp}  & \cellshb{24}{4}{20\smp}  & \cellsh{27}{5}{+}  & \cell{28}{4}{+}  & \cell{29}{3}{+}  & \cell{31}{3}{+}  & \cell{33}{3}{+}  & \cell{34}{2}{+} \\\hline
7 & \cellshbb{18}{}{8}  & \cellshb{22}{2}{20\smp}  & \cellshb{26}{4}{20\smp}  & \cellshb{30}{6}{20\smp}  & \cell{32}{6}{+}  & \cell{34}{6}{+}  & \cell{34}{4}{+}  & \cell{34}{2}{+}  & \cell{36}{2}{+} \\\hline
8 & \cellshbb{20}{}{12}  & \cellshb{24}{2}{20\smp}  & \cellshb{28}{4}{20\smp}  & \cellshb{32}{6}{20\smp}  & \cellsh{34}{6}{+}  & \cell{35}{5}{+}  & \cell{37}{5}{+}  & \cell{36}{2}{+}  & \cell{38}{2}{+} \\\hline
9 & \cellshbb{22}{}{14}  & \cellshb{26}{2}{20\smp}  & \cellshb{30}{4}{20\smp}  & \cellshb{34}{6}{20\smp}  & \cellshb{38}{8}{20\smp}  & \cell{38}{6}{+}  & \cell{40}{6}{+}  & \cell{42}{6}{+}  & \cell{40}{2}{+} \\\hline
10 & \cellshb{24}{}{20\smp}  & \cellshb{28}{2}{20\smp}  & \cellshb{32}{4}{20\smp}  & \cellshb{36}{6}{20\smp}  & \cellshb{40}{8}{20\smp}  & \cellsh{41}{7}{+}  & \cell{42}{6}{+}  & \cell{43}{5}{+}  & \cell{44}{4}{+} \\\hline
11 & \cellshb{26}{}{20\smp}  & \cellshb{30}{2}{+}  & \cellshb{34}{4}{+}  & \cellshb{38}{6}{+}  & \cellshb{42}{8}{+}  & \cellsh{44}{8}{+}  & \cell{44}{6}{+}  & \cell{46}{6}{+}  & \cell{46}{4}{+} \\\hline\end{tabular}
\caption{$n(r,z,4)$.}\label{tab:g=4}
\end{table}

\begin{table}[!ht]
\begin{minipage}[t]{4cm}
\begin{tabular}[t]{|>{\columncolor{azulito}}r||r|r|r|r|r|r|r|r|r|}
\hline
\rowcolor{azulito}
$r\backslash z$ &1 &2 &3 &4 &5  &6 \\ \hline\hline
1 & \cellbb{8}{}{1}  & \cellbb{10}{}{1}  & \cellbb{16}{4}{1}  & \cellbb{18}{4}{1}  & \cell{24}{8}{+}  & \cell{24}{6}{+} \\\hline
2 & \cellbb{13}{}{1}  & \cellbb{19}{4}{1}  & \cell{24}{7}{+}  & \cellb{25}{6}{+}  & \cellb{25}{4}{1\smp}  & \cell{27}{4}{+} \\\hline
3 & \cellbb{24}{4}{23}  & \cell{28}{6}{+}  & \cell{32}{8}{+}  & \cell{30}{4}{+}  & \cell{30}{2}{+}  & \cell{32}{2}{+} \\\hline
4 & \cell{31}{2}{+}  & \cell{37}{6}{+}  & \cell{44}{11}{+}  & \cell{42}{7}{+}  & \cell{39}{2}{+}  & \cell{41}{2}{+} \\\hline
5 & \cell{42}{2}{+}  & \cell{50}{8}{+}  & \cell{58}{14}{+}  & \cell{60}{14}{+}  & \cell{58}{10}{+}  & \cell{52}{2}{+} \\\hline
6 & \cell{54}{1}{+}  & \cell{65}{10}{+}  & \cell{74}{17}{+}  & \cell{74}{15}{+}  & \cell{74}{13}{+}  & \cell{67}{4}{+} \\\hline
\end{tabular}
\captionof{table}{$n(r,z,5)$.}\label{tab:g=5}
\end{minipage}\hspace{3cm}%
\begin{minipage}[t]{4cm}
\begin{tabular}[t]{|>{\columncolor{azulito}}r||r|r|r|r|r|}
\hline
\rowcolor{azulito}
$r\backslash z$ &1 &2 &3 &4  &5  \\ \hline\hline
1 & \cellb{10}{}{1}  & \cellb{12}{}{1}  & \cellb{20}{6}{1\smp}  & \cellb{22}{6}{1\smp}  & \cell{24}{6}{+} \\\hline
2 & \cellb{18}{}{3}  & \cellb{27}{7}{1}  & \cell{27}{5}{+}  & \cell{29}{5}{+}  & \cell{30}{4}{+} \\\hline
3 & \cellb{30}{}{1}  & \cell{38}{6}{+}  & \cell{40}{6}{+}  & \cell{38}{2}{+}  & \cell{40}{2}{+} \\\hline
4 & \cellb{48}{2}{2\smp}  & \cell{54}{6}{+}  & \cell{62}{12}{+}  & \cell{61}{9}{+}  & \cell{55}{1}{+} \\\hline
5 & \cell{68}{2}{+}  & \cell{76}{8}{+}  & \cell{86}{16}{+}  & \cell{86}{14}{+}  & \cell{82}{8}{+} \\\hline
6 & \cellb{90}{}{1}  & \cell{102}{10}{+}  & \cell{108}{14}{+}  & \cell{109}{13}{+}  & \cell{108}{10}{+} \\\hline
7 & \cell{118}{}{+}  & \cell{132}{12}{+}  & \cell{138}{16}{+}  & \cell{138}{14}{+}  & \cell{136}{10}{+} \\\hline
\end{tabular}
\captionof{table}{$n(r,z,6)$.}\label{tab:g=6}
\end{minipage}
\end{table}

\begin{table}[!ht]
\begin{minipage}[t]{4cm}
\begin{tabular}[t]{|>{\columncolor{azulito}}r||r|r|r|r|}
\hline
\rowcolor{azulito}
$r\backslash z$ &1 &2 &3 &4\\ \hline\hline
1 & \cellbb{12}{}{1}  & \cellbb{14}{}{1}  & \cell{24}{8}{+}  & \cell{24}{6}{+} \\\hline
2 & \cellbb{25}{}{1}  & \cell{33}{6}{+}  & \cell{33}{4}{+}  & \cell{35}{4}{+} \\\hline
3 & \cell{56}{4}{+}  & \cell{60}{6}{+}  & \cell{62}{6}{+}  & \cell{62}{4}{+} \\\hline
4 & \cell{100}{1}{+}  & \cell{111}{10}{+}  & \cell{115}{12}{+}  & \cell{116}{11}{+} \\\hline
\end{tabular}
\captionof{table}{$n(r,z,7)$.}\label{tab:g=7}
\end{minipage}\hspace{3cm}%
\begin{minipage}[t]{4cm}
\begin{tabular}[t]{|>{\columncolor{azulito}}r||r|r|r|r|}
\hline
\rowcolor{azulito}
$r\backslash z$ &1 &2 &3 &4 \\ \hline\hline
1 & \cellb{14}{}{1}  & \cellb{16}{}{1}  & \cell{24}{6}{+}  & \cell{26}{6}{+} \\\hline
2 & \cellb{32}{}{4}  & \cell{40}{6}{+}  & \cell{40}{4}{+}  & \cell{42}{4}{+} \\\hline
3 & \cellb{76}{2}{1\smp}  & \cell{82}{6}{+}  & \cell{84}{6}{+}  & \cell{84}{4}{+} \\\hline
4 & \cell{152}{}{+}  & \cell{163}{9}{+}  & \cell{167}{11}{+}  & \cell{168}{10}{+} \\\hline
\end{tabular}
\captionof{table}{$n(r,z,8)$.}\label{tab:g=8}
\end{minipage}
\end{table}

\begin{table}[!ht]
\begin{minipage}[t]{4cm}
\begin{tabular}[t]{|>{\columncolor{azulito}}r||r|r|r|}
\hline
\rowcolor{azulito}
$r\backslash z$ &1 &2  &3 \\ \hline\hline
1 & \cellb{16}{}{1}  & \cellb{18}{}{1}  & \cell{26}{6}{+} \\\hline
2 & \cellb{41}{}{1}  & \cell{51}{8}{+}  & \cell{49}{4}{+} \\\hline
3 & \cell{120}{}{+}  & \cell{130}{8}{+}  & \cell{132}{8}{+} \\\hline
\end{tabular}
\captionof{table}{$n(r,z,9)$.}\label{tab:g=9}
\end{minipage}\hspace{3cm}%
\begin{minipage}[t]{4cm}
\begin{tabular}[t]{|>{\columncolor{azulito}}r||r|r|r|}
\hline
\rowcolor{azulito}
$r\backslash z$ &1 &2 &3   \\ \hline\hline
1 & \cellb{18}{}{1}  & \cellb{20}{}{1}  & \cell{28}{6}{+} \\\hline
2 & \cellb{50}{}{5}  & \cell{58}{6}{+}  & \cell{58}{4}{+} \\\hline
3 & \cell{166}{}{+}  & \cell{176}{8}{+}  & \cell{178}{8}{+} \\\hline
\end{tabular}
\captionof{table}{$n(r,z,10)$.}\label{tab:g=10}
\end{minipage}
\end{table}
\clearpage\small
\providecommand{\nolinkurl}[1]{\texttt{\detokenize{#1}}}
\providecommand{\url}[1]{\nolinkurl{#1}}
\providecommand{\mapurl}[1]{\ifx\href\undefined\url{#1}\else\href{#1}{\nolinkurl{#1}}\fi}
\providecommand{\mapdoi}[1]{\ifx\href\undefined
  doi:\nolinkurl{#1}\else\href{https://doi.org/#1}{doi:\nolinkurl{#1}}\fi}

\end{document}